\documentclass[12pt,oneside]{amsart}
\usepackage{amscd,amsfonts,amsmath,amsthm,amssymb,verbatim,mathrsfs}
\usepackage[dvips]{graphicx}
\usepackage{tikz}
\usetikzlibrary{shapes,backgrounds}
\usetikzlibrary{calc}
\usetikzlibrary{matrix}
\usepackage{graphics,hyperref}
\usepackage{enumitem}
\usepackage{psfrag}
\usepackage{latexsym}
\usepackage{pstcol,pst-plot,pst-3d}
\usepackage{multicol}
\usepackage[normalem]{ulem} 

\psset{unit=0.7cm,linewidth=0.8pt,arrowsize=2.5pt 4}

\newpsstyle{fatline}{linewidth=1.5pt}
\newpsstyle{fyp}{fillstyle=solid,fillcolor=verylight}
\definecolor{verylight}{gray}{0.97}
\definecolor{light}{gray}{0.9}
\definecolor{medium}{gray}{0.85}

\def\NZQ{\Bbb}               
\def\NN{{\NZQ N}}
\def\QQ{{\NZQ Q}}
\def\ZZ{{\NZQ Z}}

\def\AA{{\NZQ A}}
\def\PP{{\NZQ P}}
\def\FF{{\NZQ F}}

\def\frk{\frak}               

\def\Phi{{\frk n}}
\def\Phi{{\frk N}}
\def\MM{{\mathcal M}}

\def\ME{{\mathcal E}}

\def\MF{{\mathcal F}}

\def\MH{{\mathcal H}}
\def\MS{{\mathcal S}}
\def\MU{{\mathcal U}}

\def\Gr{{\mathcal Gr}}
\def\ub{{\bold u}}

\def\ab{{\bold a}}

\def\cb{{\bold c}}

\def\opn#1#2{\def#1{\operatorname{#2}}} 
\opn\chara{char} \opn\length{\ell} \opn\pd{pd} \opn\rk{rk}
\opn\projdim{proj\,dim} \opn\injdim{inj\,dim} \opn\rank{rank}
\opn\depth{depth} \opn\grade{grade} \opn\height{height}
\opn\embdim{emb\,dim} \opn\codim{codim} \opn\sgn{sgn}

\opn\Tr{Tr} \opn\bigrank{big\,rank}
\opn\superheight{superheight}\opn\lcm{lcm}
\opn\trdeg{tr\,deg}
\opn\reg{reg} \opn\lreg{lreg} \opn\ini{in} \opn\lpd{lpd}
\opn\size{size}\opn\bigsize{bigsize}
\opn\cosize{cosize}\opn\bigcosize{bigcosize}
\opn\sdepth{sdepth}\opn\sreg{sreg}
\opn\link{link}\opn\fdepth{fdepth} \opn\generic{generic}
\opn\div{div} \opn\Div{Div} \opn\cl{cl} \opn\Cl{Cl} \opn\Cor{Cor}
\opn\Spec{Spec} \opn\Supp{Supp} \opn\supp{supp} \opn\Sing{Sing}
\opn\Ass{Ass} \opn\Min{Min}\opn\Mon{Mon} \opn\dstab{dstab} \opn\astab{astab}
\opn\Ann{Ann} \opn\Rad{Rad} \opn\Soc{Soc} 
\opn\Im{Im} \opn\Ker{Ker} \opn\Coker{Coker} \opn\Am{Am}
\opn\Hom{Hom} \opn\Tor{Tor} \opn\Ext{Ext} \opn\End{End}
\opn\Aut{Aut} \opn\id{id} \opn\span{span}

\opn\nat{nat}
\opn\pff{pf}
\opn\Pf{Pf} \opn\GL{GL} \opn\SL{SL} \opn\mod{mod} \opn\ord{ord}
\opn\Gin{Gin} \opn\Hilb{Hilb}\opn\sort{sort}
\opn\aff{aff} \opn\con{conv} \opn\relint{relint} \opn\st{st}
\opn\lk{lk} \opn\cn{cn} \opn\core{core} \opn\vol{vol}
\opn\link{link} \opn\star{star}\opn\lex{lex} 
\opn\gr{gr}

\def\pot#1#2{#1[\kern-0.28ex[#2]\kern-0.28ex]}

\opn\dirlim{\underrightarrow{\lim}}
\opn\inivlim{\underleftarrow{\lim}}
\let\to=\rightarrow

\def\Implies{\ifmmode\Longrightarrow \else
        \unskip${}\Longrightarrow{}$\ignorespaces\fi}
\def\implies{\ifmmode\Rightarrow \else
        \unskip${}\Rightarrow{}$\ignorespaces\fi}
\def\iff{\ifmmode\Longleftrightarrow \else
        \unskip${}\Longleftrightarrow{}$\ignorespaces\fi}

\let\:=\colon
\newtheorem{Theorem}{Theorem}[section]
\newtheorem{Lemma}[Theorem]{Lemma}
\newtheorem{Corollary}[Theorem]{Corollary}
\newtheorem{Proposition}[Theorem]{Proposition}
\newtheorem{Remark}[Theorem]{Remark}

\newtheorem{Example}[Theorem]{Example}

\newtheorem{Definition}[Theorem]{Definition}

\newtheorem{Question}[Theorem]{Question}

\let\epsilon\varepsilon
\let\kappa=\varkappa
\makeatletter

\def\qed{\ifhmode\textqed\fi
      \ifmmode\ifinner\quad\qedsymbol\else\dispqed\fi\fi}
\def\textqed{\unskip\nobreak\penalty50
       \hskip2em\hbox{}\nobreak\hfil\qedsymbol
       \parfillskip=0pt \finalhyphendemerits=0}
\def\dispqed{\rlap{\qquad\qedsymbol}}

\opn\dis{dis}
\def\pnt{{\raise0.5mm\hbox{\large\bf.}}}

\opn\Lex{Lex}

\begin{document}


\title{ Bouquet algebra of toric ideals }

\author{Sonja Petrovi\'c, Apostolos Thoma, Marius Vladoiu}\thanks{The third author was partially supported by project PN-II-RU-TE-2012-3-0161, granted by the Romanian National Authority for Scientific Research, CNCS - UEFISCDI}

\address{Sonja Petrovi\'c, Department of Applied Mathematics,  Illinois Institute of Technology, Chicago 60616, USA}
\email{sonja.petrovic@iit.edu}

\address{Apostolos Thoma, Department of Mathematics, University of Ioannina, Ioannina 45110, Greece} 
\email{athoma@uoi.gr}


\address{Marius Vladoiu, Faculty of Mathematics and Computer Science, University of Bucharest, Str. Academiei 14, Bucharest, RO-010014, Romania, and}
\address{Simion Stoilow Institute of Mathematics of Romanian Academy, Research group of the project PN-II-RU-TE-2012-3-0161, P.O.Box 1--764, Bucharest 014700, Romania}
\email{vladoiu@fmi.unibuc.ro}


\begin{abstract}
To any toric ideal $I_A$, encoded by an integer matrix $A$, we associate a matroid structure called {\em the bouquet graph} of $A$ and introduce another toric ideal called {\em the bouquet ideal} of $A$. We show how these   objects capture the essential combinatorial and algebraic information  about  $I_A$. Passing from the toric ideal to its bouquet ideal reveals a structure that allows us to classify several cases. For example, on the one end of the spectrum, there are ideals that we call {\em stable}, for which bouquets capture  the complexity of  various generating sets as well as  the minimal free resolution. On the other end of the spectrum lie toric ideals whose various bases (e.g., minimal generating sets, Gr\"obner, Graver bases) coincide. Apart from  allowing for classification-type results, bouquets provide a new way to construct families of examples of toric ideals with various interesting properties, such as robustness, genericity, and unimodularity.   The new bouquet framework can be used to provide a characterization of toric ideals  whose  Graver basis, the universal Gr\"obner  basis, any reduced Gr\"obner basis and any minimal generating set  coincide.  
\end{abstract}
\maketitle

\section*{Introduction} 

Toric ideals appear prominently in polyhedral geometry, algebraic topology, and algebraic geometry. Naturally,  most famous classes of toric ideals come equipped with a rich algebraic and homological structure. They also have a common combinatorial feature, namely, equality of various bases. For example, generic toric ideals are minimally generated  by indispensable binomials, robust toric by the universal Gr\"obner basis, those that are Lawrence  are minimally generated by the Graver basis, and circuits equal the Graver basis for unimodular toric ideals.  

In light of this, the present  manuscript offers a combinatorial classification of all toric ideals. This classification simultaneously reveals equality of various distinguished subsets of binomials, provides a unifying framework for studying combinatorial signatures of toric ideals, and introduces a technique to solve several related (open) problems from combinatorial commutative algebra, algebraic geometry, combinatorics and integer programming. Furthermore, it provides a technique to construct (infinitely many) examples of five important classes of toric ideals (which turned out to be a challenge so far for generic, robust and strongly robust) often exploited in proving several important results in our paper. Before stating our main results, Theorems A-E below, we offer a brief overview of the motivation and some consequences. 

As is customary in the literature, we consider toric ideals to be encoded by a matrix $A$ whose columns are exponents in the monomial parametrization of the corresponding toric variety. Throughout the paper, we also refer to  bases of $I_A$ as  those of $A$. 

\subsection*{Various bases of  toric ideals} Apart from minimal generating sets, which we may also refer to as minimal Markov bases for brevity {\cite{DSS,AHT}, the well-known distinguished subsets of binomials in a toric ideal $I_A$ are: the set of circuits $\mathcal C(A)$, the universal Gr\"obner basis $\mathcal U(A)$, the set of indispensable binomials $\MS(A)$, the universal Markov basis $\MM(A)$, and  the Graver basis $\Gr(A)$. Briefly, $\MS(A)$ are binomials appearing in every minimal generating set; $\mathcal U(A)$ is the union of all of the (finitely many) reduced Gr\"obner bases; and $\Gr(A)$  is a crucial set for integer programming \cite{G,DHK} consisting of the primitive binomials in the ideal.  When $\Ker_{\ZZ}(A)\cap\NN^n=\{\bf 0\}$, the generally proper inclusions between them hold as follows. 
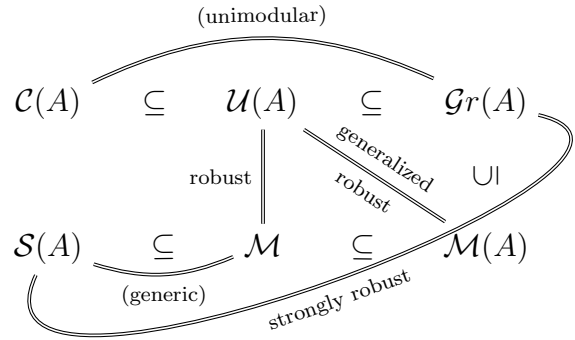
\begin{center}
\begin{tikzpicture}[descr/.style={fill=white,inner sep=1.5pt}]
    \matrix (m) [matrix of math nodes,row sep=3em,column sep=4em]
  {
     \mathcal C(A) & \mathcal U(A) & \Gr(A) \\
     \MS(A) & \MM & \MM(A) \\};
     \path[-stealth]
    (m-1-1) edge [draw=none] node [sloped] {$\subseteq$} (m-1-2)
    (m-1-2) edge [draw=none] node [sloped] {$\subseteq$} (m-1-3)
    (m-2-1) edge [draw=none] node [sloped] {$\subseteq$} (m-2-2)  
    (m-2-2) edge [draw=none] node [sloped] {$\subseteq$} (m-2-3) 
    (m-1-2) edge [double,-] node [left] {\tiny{robust}} (m-2-2)
    (m-1-2) edge [double,-] node [above, sloped] {\tiny{generalized}} 
                            node [below, sloped] {\tiny{robust}} (m-2-3) 
    (m-1-3) edge[double,-,out=-13,in=245] node[below, sloped] {\tiny{ strongly robust}} (m-2-1) 
    (m-1-1) edge[double,-,bend left=25] node[above,sloped] {\tiny{(unimodular)}} (m-1-3) 
    (m-2-1) edge[double,-,bend right=20] node[below] {\tiny{(generic)}} (m-2-2)  
    (m-2-3) edge[draw=none] node [sloped, auto=false, allow upside down] {$\subseteq$} (m-1-3);    
\end{tikzpicture}
\end{center}
\vspace{-1.5cm}
In the diagram above, $\MM$ represents one minimal binomial generating set of $I_A$, or, in the terminology used later in the paper and in some applications, a minimal Markov basis. For a discussion on the inclusions in the general case, see \cite{CTV2}. 

Equalities between some pairs of these sets were known previously in  special cases, perhaps the best known of which is that if $A$ is unimodular then $\mathcal C(A)=\Gr(A)$. One new class we introduce is  $S$-Lawrence ideals, playing a prominent role in our classification; see Theorem C below.  

\subsection*{Five prominent classes of toric ideals} As mentioned before, apart from classifying toric ideals according to their combinatorial signatures, we provide a technique that constructs infinitely many examples of toric ideals that are of interest in various fields. Let us outline their roles and definitions. 

\emph{Robust} toric ideals are those for which the universal Gr\"obner basis is a minimal generating set. Unlike the other four classes, robustness makes sense  for arbitrary polynomial ideals as well. A beautiful family of robust ideals are ideals of maximal minors of matrices of linear forms that are column-graded  \cite[Theorem 4.2]{CDG}, which generalize the ideals of maximal minors of a matrix of indeterminates \cite{BZ,SZ}, and the ideals of maximal minors of a sparse matrix of indeterminates \cite{Bo}. Another family of robust ideals is given by the defining ideals $I(\widetilde{L})$ of the closure $\widetilde{L}$ in $(\PP^1)^n$ of linear spaces $L\subset \AA^n$, see \cite[Theorem 1.3]{AB}. 
 Restricting to special cases, of course, reveals more properties. For example, the toric ideal of a simple graph is robust if and only if is minimally generated by its Graver basis; this was shown in \cite[Theorem 3.2]{BBDLMNS}  using the graph-theoretic classification of bases from \cite{RTT,TT}. In other words, $\MS(A)=\MU(A)=\MM(A)=\Gr(A)$ holds for robust toric ideals of graphs. 
More generally, $\MS(A)=\MM(A)$ holds for any robust toric ideal, see \cite[Theorem 5.10]{T}, and thus  $\MU(A)=\MM(A)$ also. 

In the class of \emph{generalized robust} toric ideals, introduced in \cite{T},  the universal Markov basis equals the universal Gr\"obner basis.  One very nice property of this class of ideals is that it properly includes the class of robust toric ideals, see \cite[Corollary 5.12]{T} and \cite[Example 3.6]{T}.  
The third class are toric ideals whose circuits equal the Graver basis. A proper subclass is the set of unimodular toric ideals, where equality of these bases has numerous and beautiful consequences in hyperplane arrangements, resolutions, spectral sequences, and algebraic geometry, see \cite{BPS}. The fourth class are toric ideals whose indispensable binomials form a Graver basis; this class  is named by Sullivant \emph{strongly robust}, see \cite{Su}, since a priori it is just a subclass of robust toric ideals. The more technical name for strongly robust ideals, which are introduced in Section~\ref{sec:EqualBases}, is \emph{$\emptyset$-Lawrence}. However, it is conjectured that all robust toric ideals are minimally generated by their Graver bases, see \cite[Question 6.1]{BBDLMNS}, in other words, that all robust toric ideals are strongly robust. As we will see below, bouquets are exactly what can be used to classify toric ideals that belong to this class. Here, all  five sets  of binomials  from the diagram except $\mathcal C(A)$ are equal, which is expected as second Lawrence liftings are a special case. Finally, the fifth class of toric ideals are those for which  the set of indispensable binomials equals the universal Markov basis, that is, they have a unique minimal system of generators. A proper subclass of these are generic toric ideals \cite{PS}, where all binomials in a minimal generating set have full support.

\subsection*{The fundamental construction: bouquets} 
Every integer matrix $A$ comes equipped with a natural oriented matroid structure called the \emph{bouquet graph} of $A$, see Definition~\ref{bouq_def}. Connected components $B_1,\ldots,B_s$ of this graph are the bouquets of $A$ (and, in addition, we distinguish different types of bouquets: free or non-free, and the latter can be mixed or non-mixed; these are technical definitions that are not difficult to check and play a crucial role in the paper). Bouquets are encoded by vectors that, essentially, record the dependencies from the Gale transform of $A$: the bouquet-index-encoding vectors $\cb_{B_i}$ and the vectors $\ab_{B_i}$ whose columns make up the defining matrix  $A_B$ (cf. Definition~\ref{defn:a_b}) of the \emph{bouquet ideal} associated to $I_A$.
This new toric  ideal  encodes the basic properties of $I_A$ via a structural decomposition of $A$. 
The terminology we use comes from the applications to hypergraphs (see  \cite[Section 2]{PTV}). The graph's connected components alone were used  in \cite{BDR}, under the name \emph{coparallelism classes}, to provide combinatorial descriptions of self-dual projective toric varieties  associated to a non-pyramidal configuration, see \cite[Theorem 4.16]{BDR}. However, the extra information we give on the types of bouquets was essential in \cite{TV} to give more insight on the combinatorics of the self-dual projective toric varieties and to describe them completely.

We ask -- and answer in various ways -- the following question:  What does the bouquet structure of $A$ say about the toric ideal? In particular, we are interested in how bouquets preserve three types of properties of a toric ideal: 1) ``all'' of its combinatorial properties, that is, the structure of $\Ker_{\ZZ}(A)$, Gr\"obner, Graver, Markov bases, circuits, indispensable binomials, etc.; 2)  ``essential" combinatorics, that is,  $\Ker_{\ZZ}(A)$, the Graver basis, and  the circuits; and 3) homological information, i.e., the minimal free resolution. The following result summarizes how bouquets preserve the essential combinatorics of toric ideals. 

\smallskip
\noindent\textbf{Theorem A} (Theorems \ref{gen_iso_kernels} and \ref{gen_all_is_well}) \textit{Let $A=[{\bf a}_1,\dots,{\bf a}_n]\in\ZZ^{ m\times n}$ and its bouquet matrix $A_B=[\ab_{B_1},\dots,\ab_{B_s}]$. There is a bijective correspondence between the elements  
of $\Ker_{\ZZ}(A)$ in general,  and $\Gr(A)$ and $\mathcal C(A)$ in particular,  and the elements of $\Ker_{\ZZ}(A_B)$, and $\Gr(A_B)$ and $\mathcal C(A_B)$, respectively. More precisely, this correspondence is defined as follows: for ${\bf u}=(u_1,\ldots,u_{s})\in\Ker_{\ZZ}(A_B)$ then $B({\bf u})={\bf c}_{B_1}u_1+\cdots+{\bf c}_{B_s}u_s\in\Ker_{\ZZ}(A)$. }
\smallskip

In particular, Theorem A solves \cite[Problem 6.3]{PeSt} for an arbitrary toric ideal; the problem was posed for 0/1 matrices.  The basic idea is that the bouquet construction gives a way for recovering  all Graver bases elements from the Graver bases elements of a matrix with possibly fewer columns than $A$.  Replacing bouquets with subbouquets, see Remark~\ref{subbouquets} for definition, Theorem A is still valid and used very often throughout the paper. Since any (sub)bouquet $B$ of a matrix $A$ is encoded by two vectors $\ab_B$ and $\cb_B$ it is natural to ask whether for a given set of vectors $\ab_i$ and $\cb_i$ that can act as bouquet-index-encoding vectors there exists a matrix $A$  whose (sub)bouquets are encoded by the given set of vectors. The answer is yes and given by the following:        

\smallskip
\noindent\textbf{Theorem B} (Theorem~\ref{inverse_construction})
Let $\{\ab_1,\ldots,\ab_s\}\subset\ZZ^m$ be an arbitrary set of vectors. Let $\cb_1,\ldots,\cb_s$ be any set of primitive vectors, with $\cb_i\in\ZZ^{m_i}$ for some $m_i\geq 1$, each having full support and a positive first coordinate.  Then, there exists a matrix $A$ with the subbouquet decomposition $B_1,\ldots,B_s$, such that the $i^{th}$ subbouquet	is encoded by the following vectors: $\ab_{B_i}=(\ab_i,{\bf 0},\ldots,{\bf 0})$ and $\cb_{B_i}=({\bf 0},\ldots,\cb_i,\ldots,{\bf 0})$, where the support of $\cb_{B_i}$ is precisely in the $i^{th}$ block, of size $m_i$. 
\smallskip

One natural application of this inverse construction is that provides via Theorem A infinitely many examples of any of the 5 special classes of toric ideals from the diagram. In addition, based on the structural result on bouquets captured in Theorem A and Corollary~\ref{gen_no_free} of Theorem D, the inverse construction from  Theorem B is the fundamental tool  in \cite{TV} for describing all the defining matrices of self-dual projective toric varieties.

\subsection*{Combinatorial classification and some consequences}
As mentioned briefly above, non-free bouquets can be mixed or non-mixed, 
 and of course a toric ideal can have both in its bouquet graph. 
At one end of the spectrum 
 are the toric ideals with all of the non-free bouquets non-mixed; we call these ideals \emph{stable}. The notion of stability is quite nice, as it captures the case when passing from $I_A$ to $I_{A_B}$  preserves all combinatorial information:
  
\smallskip
\noindent\textbf{Theorem C} (Theorem \ref{stable_toric}) \textit{Let $I_A$ be a stable toric ideal. Then the bijective correspondence between the elements of $\Ker_{\ZZ}(A)$ and $\Ker_{\ZZ}(A_B)$ given by ${\bf u}\mapsto B({\bf u})$, is preserved when we restrict to any of the following sets: Graver basis, circuits, indispensable binomials, minimal Markov bases, reduced Gr\"obner bases (universal Gr\"obner basis).}
\smallskip

\noindent Furthermore, in the positively graded case, even the homological information 
is preserved; see Section~\ref{sec:OnStableToricIdeals}, which  is dedicated to stable toric ideals,  and Theorem~\ref{all_homological_stable} in particular. In addition, combining it with Theorem B gives, to our knowledge, a first way of constructing infinite classes of generic toric ideals, providing a partial answer to the open question posed at the end of \cite[Section 9.4]{MS} to find a deterministic construction of generic lattice ideals with prescribed properties, see Theorem~\ref{generic_preserved} and Remark~\ref{generic_construct}.

\smallskip
At the other end of the spectrum of the classification are toric ideals all of whose bouquets are mixed.  
In particular, such toric ideals are $\emptyset$-Lawrence
; however,  the  converse does not hold: there exist $\emptyset$-Lawrence toric ideals that have both mixed and non-mixed bouquets. 
Thus to capture the correct property we are interested in, we introduce the \emph{$S$-Lawrence property} that, intuitively, ``balances" from a trivial condition, namely being $[n]$-Lawrence, common to all toric ideals in $K[x_1,\ldots,x_n]$, to a very special class of ideals, those that are $\emptyset$-Lawrence or, in other words, strongly robust. The main result in this direction explains how mixed bouquets capture essential combinatorics of $A$. 

\smallskip
\noindent\textbf{Theorem D} (Theorem~\ref{graver=indisp}) {\it Let $B_1, \dots, B_s$ be the bouquets of $A=[{\bf a}_1,\dots,{\bf a}_n]$, and define $A_B=[\ab_{B_1},\dots,\ab_{B_s}]$. Let $S\subset [s]$ be the subset of coordinates corresponding to the mixed bouquets. Then the following are equivalent:
\\\noindent $(a)$ There exists no element in the Graver basis of $A_B$ which has a proper semiconformal decomposition that is conformal on the coordinates corresponding to $S$; 
\\\noindent $(b)$ The toric ideal of $A$ is  strongly robust;
\\\noindent $(c)$ The toric ideal of $A_B$ is $S$-Lawrence. }
\smallskip

\noindent In particular, all five sets of $I_A$ from the diagram except $\mathcal C(A)$}  coincide  if and only if the toric ideal of $A_B$ is $S$-Lawrence (Definition~\ref{S-Lawrence}), where $S$ is encoded by coordinates of the mixed bouquets of $A$. Note  that the equality $\MS(A)=\mathcal U(A)=\MM(A)=\Gr(A)$ provides, in particular, a combinatorial characterization of toric ideals for which  $\MS(A)$ and $\Gr(A)$ are equal. The question about this equality is a long-standing open problem, although  many examples have been discovered \cite{BPS,BHP,CTV2,St}.  Note that the combinatorial characterization provided by Theorem D describes a subclass of robust toric ideals or in the best case, if the answer of \cite[Question 6.1]{BBDLMNS} were positive, then it would be valid for all robust toric ideals. On the other hand, understanding better the strongly robust toric ideals seems to be closely related to their bouquet structure, since all the examples known so far have mixed bouquets. Computational evidence and partial positive answers to Question~\ref{strongly_mixed} indicate that it might happen that a toric ideal cannot be strongly robust unless its bouquet structure contains at least one mixed bouquet.


\subsection*{Applications of our results}
It is natural to ask how our results, which are valid for arbitrary integer matrices, specialize when restricting to $0/1$ matrices, that is, incidence matrices of hypergraphs. Such interest is justified by the special place of $0/1$ matrices in the world of toric ideals due to their  applications in various fields. To this end, \cite{PTV} provide several results. Namely, they identify some special types of bouquets for $0/1$ matrices, called bouquets with bases (see \cite[Definition 2.1]{PTV} for details). Applying the results from this paper, one can obtain a surprisingly non-obvious general behavior of hypergraphs \cite[Theorem 3.2]{PTV}. This result has nice consequences in integer programming, implying two universality results \cite[Corollaries 3.3, 3.4]{PTV} about the unboundedness of the degrees of all elements in a minimal Markov basis (and similarly about Graver basis, universal Gr\" obner basis, indispensable binomials respectively). Finally, the most important consequence of our results is a kind of polarization-type operation, see \cite[Theorem 4.2]{PTV}, which allows us, in particular, to answer problems like \cite[Question 6.1]{BBDLMNS}, for arbitrary toric ideals, by reducing them to toric ideals of hypergraphs. Details of those results are beyond the scope of the present manuscript. 
 
{\bf Acknowledgments.} The authors would like to thank Stavros Papadakis for pointing us to Theorem~\ref{all_homological_stable}, whose proof is similar to \cite[Proposition 6.5]{BP}.

 \section{ Bouquet decomposition of a toric ideal}
\label{sec:BouqDec} 
 
 
 
 Let $K$ be a field and $A=[\ab_1,\dots,\ab_n]\in\ZZ^{m\times n}$ be an integer matrix with  columns $\{\ab_i\}$. We recall that the toric ideal of $A$ is the ideal $I_A\subset K[x_1,\ldots,x_n]$ given by $$I_A=({\bf x}^{{\bf u}^+}-{\bf x}^{{\bf u}^-}: {\bf u}\in\Ker_{\ZZ}(A)),$$
where ${\bf u}={\bf u}^+-{\bf u}^-$ is the unique expression of an integral vector ${\bf u}$ as a difference of two non-negative integral vectors with disjoint support, see \cite[Section 4]{St}. As usual, we denote by ${\bf x}^{{\bf u}}$ the monomial $x_1^{u_1}\cdots x_n^{u_n}$, with ${\bf u}=(u_1,\ldots,u_n)$. Denote by $r$ the $\dim_{\QQ} \QQ(A)$. Fix a basis $G_1, G_2, \dots, G_{n-r}\in\ZZ^n$ for the lattice $\Ker_{\ZZ}(A)$, and denote by  $G$  the $n\times (n-r)$ matrix with columns $G_1,\dots,G_{n-r}$.  Define $G(A)=\{{\bf b}_1, \dots, {\bf b}_n\}$ to be the set of ordered rows of $G$. The set $G(A)$ is called the {\em Gale transform} of $A$, while the vector $G({\bf a}_i):={\bf b}_i$ is called the Gale transform of ${\bf a}_i$ for any $i$. 
 
To the columns of $A$ one can associate the oriented vector matroid $M_A$ (see \cite{BLSWZ} for details). The support of a vector ${\bf u}\in \ZZ^n$ is the set supp$({\bf u})=\{i|u_i\not=0\}\subset \{1, \dots, n\}$. A {\em co-vector} is any  vector of the form $({\bf u} \cdot {\bf a}_1, \dots,{\bf u} \cdot {\bf a}_n)$. A {\em co-circuit} of $A$ is any non-zero co-vector of minimal support.  A co-circuit with support of cardinality one is called a {\em co-loop}. We call the vector ${\bf a}_i$ {\em free}  if $\{i\}$ is the support of a co-loop. Since a co-loop is characterized by the property that it belongs to any basis of the matroid, a free vector ${\bf a}_i$  belongs to any basis of $M_A$. Equivalently, the Gale transform of a free vector, $G({\bf a_i})$ is equal to the zero vector, which means that $i$ is not contained in the support of any minimal  generator of the toric ideal $I_A$, or any element in the Graver basis. 

 Let $E_A$ be the set consisting of elements of the form $\{ {\bf a}_i, {\bf a}_j\}$ such that there exists a co-vector ${\bf c}_{ij}$ with support  $\{i, j\}$. We denote by $E_A^+$  the subset of $E_A$ where the co-vector is a co-circuit and the signs of the two nonzero coordinates of ${\bf c}_{ij}$ are distinct, and we denote by $E_A^-$ the subset of $E_A$ where the co-vector is a co-circuit and the signs of the two nonzero coordinates of ${\bf c}_{ij}$ are the same. Furthermore, we denote by $E_A^0$ the subset of $E_A$ where the co-vector is not a co-circuit. This implies that both ${\bf a}_i$ and ${\bf a}_j$ are free vectors. By definition, these three sets $E_A^+,  E_A^-,  E_A^0$ partition $E_A$.

\begin{Definition}\label{bouq_def}
{\em The {\em bouquet graph} $G_A$ of $I_A$ is the graph whose vertex set  is $\{{\bf a}_1,\dots, {\bf a}_n\}$ and edge set  $E_A$.  The {\em bouquets} of $A$ are the connected components of $G_A$.  If there are free vectors in $A$ they form one bouquet with all edges in $E_A^0$, which we call the {\em free bouquet}  of $G_A$. A bouquet which is not free is called non-free. A non-free bouquet is called {\em mixed} if it contains at least an edge from $E_A^-$, and {\em non-mixed} if it is either an isolated vertex or all of its edges are from $E_A^+$.} 
\end{Definition}

It follows from the co-circuit axioms of  oriented matroids that each bouquet of $A$ is a clique in $G_A$. The discussions preceding Definition~\ref{bouq_def} show that the free bouquet of $G_A$ consists of all $\ab_i$ such that $G(\ab_i)={\bf 0}$. Moreover, in the following Lemma we give an equivalent description of  non-free, mixed and non-mixed bouquets of $A$ in terms of the Gale transform $G(A)$. These descriptions are based on a well-known result about the characterization of co-circuits of cardinality two in terms of Gale transforms, whose proof is included for convenience of the reader. By slight abuse of notation, we identify vertices of $G_A$ with their labels; that is, $\ab_i$ will be used to denote vectors in the context of $A$ and $M_A$, and vertices in the context of $G_A$.   

\begin{Lemma}\label{bouq_check}
Suppose that not all of the columns $\ab_i$ of $A$ 
 are free. 
Then: 
\begin{enumerate}
\item[(a)] There exists a co-circuit of cardinality two with support $\{i,j\}$ if and only if  ${\bf a}_i,{\bf a}_j$ are not free and $G({\bf a}_i)=\lambda G({\bf a}_j)$ for some $\lambda\neq 0$.
\item[(b)] $B\subset A$ is a  non-free bouquet if and only if the vector space $<G({\bf a}_i)| {\bf a}_i\in B>$ has dimension one.
\item[(c)] The edge $\{{\bf a}_i,{\bf a}_j\}$ belongs to $E_A^+$ (respectively $E_A^-$) if and only if ${\bf a}_i,{\bf a}_j$ are not free and $G({\bf a}_i)=\lambda G({\bf a}_j)$ for some $\lambda>0$ (respectively $\lambda<0$).
\end{enumerate}
\end{Lemma}
\begin{proof}
Let $G$ be the $n\times (n-r)$ matrix with the column vectors $G_1,\ldots,G_{n-r}$, and whose set of row vectors is the Gale transform $G(A)$ of $A$. For simplicity set $t=n-r$, $G({\bf a}_i)={\bf b}_i=(b_{i1},\ldots,b_{it})$ for all $i\in\{1,\ldots,n\}$, and note that ${\bf a}_i$ not being free implies $G({\bf a}_i)\neq{\bf 0}$.  Definition of $G$ provides:
\begin{eqnarray}\label{gal}
b_{1k}{\bf a}_1+\cdots+b_{nk}{\bf a}_n={\bf 0} \ \text{ for all } \  k\in\{1,\ldots,t\}.  
\end{eqnarray}
First we show (a).   Assume  that there exists a co-circuit ${\bf c}_{ij}$ of cardinality two, with support $\{i,j\}$:  ${\bf c}_{ij}=(0,\ldots,c_i,\ldots,c_j,\ldots,0)$ and $c_i,c_j\neq 0$. This implies the existence of a vector ${\bf v}\in\ZZ^m$ such that ${\bf v}\cdot {\bf a}_k=0$ for any $k\neq i,j$, ${\bf v}\cdot {\bf a}_i=c_i$ and ${\bf v}\cdot {\bf a}_j=c_j$. Taking the dot product with ${\bf v}$ in (\ref{gal}) we obtain $b_{ik}c_i+b_{jk}c_j=0$ for all $k\in\{1,\ldots,t\}$. Therefore $G({\bf a}_i)=-\frac{c_j}{c_i}G({\bf a}_j)$, the desired conclusion. 

Conversely, let $G({\bf a}_i)=\lambda G({\bf a}_j)$ for some $\lambda\neq 0$. Since ${\bf a}_i,{\bf a}_j$ are not free we also have that $G({\bf a}_i), G({\bf a}_j)\neq{\bf 0}$. It is a basic fact in matroid theory, see \cite{O}, that the co-circuits of a matroid are the minimal sets having non-empty intersection with every basis of the matroid. Thus, in order to prove the existence of a co-circuit of cardinality two with support $\{i,j\}$, it is enough to prove 
 that 1) any basis of $M_A$ contains either ${\bf a}_i$ or ${\bf a}_j$, and 2) there are no co-loops with support $\{i\}$ or $\{j\}$. The latter  is automatically satisfied since ${\bf a_i},{\bf a}_j$ are not free. Assume by contradiction that there exists a basis of $<{\bf a}_1,\ldots,{\bf a}_n>$ which does not contain both ${\bf a}_i$ and ${\bf a}_j$. This implies that ${\bf a}_i=\sum_{k\neq i,j}\beta_k {\bf a}_k$. Thus, the vector ${\bf w}$ of this relation,  whose $j$-th, $i$-th, and $k$-th (for any $k\neq i,j$) coordinates are $0$, $-1$, and $\beta_k$ respectively, is a linear combination of $G_1,\ldots,G_t$. Therefore, ${\bf w}=\sum_{l=1}^t \gamma_lG_l$, and, in particular,
\[
-1=\sum_{l=1}^t \gamma_l b_{il} \ \text{ and } \ 0=\sum_{l=1}^t \gamma_l b_{jl}. 
\]    
This yields a contradiction, since by hypothesis $b_{il}=\lambda b_{jl}$ for any $l\in\{1,\ldots,t\}$. 

Note that (b) follows immediately from (a) and the definition of a bouquet. Finally, let $\{{\bf a}_i,{\bf a}_j\}$ be an edge in  $E_A\setminus E_A^{0}$. This implies that (a) holds, and from its proof we obtain  $\lambda=-\frac{c_j}{c_i}$. Hence (c) follows at once from the definition of $E_A^+$ and $E_A^-$.  \qed

\end{proof}

 Due to its importance for later sections, we isolate the following consequence of the proof of Lemma~\ref{bouq_check}.  

\begin{Remark}\label{cocircuit_gale}
\rm Suppose that $A=[\ab_1,\ldots,\ab_n]$ has a co-circuit $(0,\ldots,c_i,\ldots,c_j,\ldots,0)$ with support of cardinality at most two. Then we have $c_iG(\ab_i)+c_jG(\ab_j)={\bf 0}$.    
\end{Remark}

In particular, it follows from Lemma~\ref{bouq_check}(c) that if $\{{\bf a}_i,{\bf a}_j\}\in E_A^+$ and $\{{\bf a}_j,{\bf a}_s\}\in E_A^+$ then $\{{\bf a}_i,{\bf a}_s\}\in E_A^+$, while if $\{{\bf a}_i,{\bf a}_j\}\in E_A^+$ and $\{{\bf a}_j,{\bf a}_s\}\in E_A^-$ then $\{{\bf a}_i,{\bf a}_s\}\in E_A^-$. Combining  these considerations provides  an algorithm for computing  the bouquet graph of  a toric ideal through the computation of the Gale transform of $A$, as the following example illustrates.  
 
 \begin{Example}\label{bouq_construction}
 \rm 
 Let $A$ be the integer matrix 
\[ 
\left( \begin{array}{ccccccc}
1 & 1 & 1 & 0 & 0 & 0 & 0\\
1 & 0 & 0 & 1 & 1 & 0 & 0\\
0 & 1 & 0 & 1 & 0 & 1 & 0\\
0 & 0 & 1 & 0 & 1 & 1 & 0\\
0 & 0 & 0 & 0 & 0 & 0 & 3
\end{array} \right),
\]
with columns $\{{\bf a}_1,\ldots,{\bf a}_7\}$.
 A basis for $\Ker_{\ZZ}(A)$ is given by $(1, -1, 0, 0, -1, 1, 0)$ and $(1, 0, -1, -1, 0, 1, 0)$, and thus the Gale transform $G(A)$ of $A$ consists of the following seven vectors: $G({\bf a}_1)=(1, 1)$, $G({\bf a}_2)=(-1, 0)$, $G({\bf a}_3)=(0, -1)$, $G({\bf a}_4)=(0, -1)$, $G({\bf a}_5)=(-1, 0)$, $G({\bf a}_6)=(1, 1)$ and $G({\bf a}_7)=(0, 0)$.  Hence ${\bf a}_7$ is a free vector. So, the graph $G_A$ has the vertex set $\{{\bf a}_1,\ldots , {\bf a}_7\}$, and applying Lemma~\ref{bouq_check}(a) and (b), we see immediately that the edges of $G_A$ are $\{{\bf a}_1,{\bf a}_6\}$, $\{{\bf a}_2,{\bf a}_5\}$ and $\{{\bf a}_3,{\bf a}_4\}$. Therefore $G_A$ has three bouquets each consisting of a single edge and one additional free bouquet consisting of the isolated vertex ${\bf a}_7$. Moreover, since $G({\bf a}_1)=G({\bf a}_6)$, $G({\bf a}_2)=G({\bf a}_5)$ and $G({\bf a}_3)=G({\bf a}_4)$,  Lemma~\ref{bouq_check}(c) provides  $\{{\bf a}_1,{\bf a}_6\}, \{{\bf a}_2,{\bf a}_5\}, \{{\bf a}_3,{\bf a}_4\}\in E_A^+$. In this case, all of the non-free bouquets are non-mixed. 
\end{Example} 

In the remainder of this section, we will show that the bouquets of $A$ determine, in a sense, elements of $I_A$. Before stating the main results, the following technical definitions are needed.

 First, we define a \emph{bouquet-index-encoding vector} $\cb_B$ as follows\footnote{Remark on notation: definition of $\cb_B$ does depend on the matrix $A$, but reference to it is suppressed for ease of notation. The underlying matrix $A$ for the bouquet $B$  will always be clear from the context.}. If the bouquet $B$ is free then we set $\cb_B\in\ZZ^n$ to be any nonzero vector such that $\supp(\cb_B)=\{i: \ab_i\in B\}$ and with the property that the first nonzero coordinate is positive. For  a  non-free bouquet $B$ of $A$, consider the Gale transforms of the elements in $B$. All the elements are nonzero and pairwise linearly dependent, therefore there exists a nonzero coordinate $j$ in all of them. Let $g_j=\gcd(G({\bf a}_i)_j| \ {\bf a}_i\in B)$ and fix the smallest integer $i_0$ such that ${\bf a}_{i_0}\in B$. Let ${\bf c}_B$  be the vector in $\ZZ^n$ whose $i$-th coordinate  is $0$ if ${\bf a}_i\notin B$, and is $\varepsilon_{i_0j}G({\bf a}_i)_j/g_j$ if ${\bf a}_i \in B$, where $\varepsilon_{i_0j}$ represents the sign of the integer $G({\bf a}_{i_0})_j$. 
 Thus the $\supp({\bf c}_B)=\{i: {\bf a}_i\in B\}$. Note that the choice of $i_0$ implies that the first nonzero coordinate of ${\bf c}_B$ is positive. Since each ${\bf a}_i$ belongs to  exactly one bouquet the supports of the vectors ${\bf c}_{B_i}$ are pairwise disjoint. In addition, $\cup_i\supp(\cb_{B_i})=[n]$.  
 
\begin{Remark}\rm
 For a non-free bouquet $B$, the definition  of  vector ${\bf c}_B$ does not depend on the nonzero coordinate $j$ chosen.  Indeed, let $k\neq j$ be such that $G({\bf a}_{i_0})_j\neq 0$ and assume that the bouquet $B$ has $s+1$ vertices ${\bf a}_{i_0},\ldots,{\bf a}_{i_s}$. For sake of simplicity denote by $c_l:=G({\bf a}_{i_l})_j$ and $d_l:=G({\bf a}_{i_l})_k$ for all $0\leq l\leq s$. Then we have to prove that $\varepsilon_{i_0j}c_l/g_j=\varepsilon_{i_0k}d_l/g_k$ for all $l$. The case $s=0$ is obvious, so we may assume that $s\geq 1$. Since $B$ is a bouquet then $c_l/d_l=m/n$ for all $l$, for some relatively prime integers $m,n$. On the other hand $g_j=\gcd(c_0,\ldots,c_s)$ implies that $g_j=\lambda_0c_0+\cdots+\lambda_sc_s$ for some integers $\lambda_0,\ldots,\lambda_s$. From here we obtain $g_j n/m=\lambda_0d_0+\cdots+\lambda_sd_s\in\ZZ$, and since $g_k=\gcd(d_0,\ldots,d_s)$ then $g_k|g_j(n/m)$. Similarly, we obtain that $g_j|g_k(m/n)$ and consequently $g_jn=\pm g_km$, which implies $c_l/d_l=m/n=\pm g_j/g_k$. Since $g_j,g_k>0$, it is easy to see that $(\varepsilon_{i_0j}c_l)/(\varepsilon_{i_0k}d_l)=g_j/g_k$, and therefore we obtain the desired equality. 
  \end{Remark}
Whether the bouquet $B$ is mixed or not can be read off from the vector $\cb_B$ as follows. 
 
\begin{Lemma}\label{mixed_bouquet}
Suppose $B$ is a  non-free bouquet of $A$. 
Then $B$ is a mixed bouquet if and only if the vector ${\bf c}_B$ has a negative and a positive coordinate. 
\end{Lemma}
\begin{proof}
Let $B$ be a mixed bouquet and assume without loss of generality that ${\bf a}_1,{\bf a}_2\in B$ such that $\{{\bf a}_1,{\bf a}_2\}\in E_A^-$. Applying Lemma~\ref{bouq_check}(c) there exists $\lambda<0$ such that $G({\bf a}_{1})=\lambda G({\bf a}_{2})$. Since ${\bf a}_1\in B$ then ${\bf a}_1$ is not a free vector and thus there exists an integer $j$ such that $G({\bf a}_{1})_j\neq 0$. Then $$({c}_B)_1({c}_B)_2=\varepsilon_{1j}\frac{G({\bf a}_1)_j}{g_j}\varepsilon_{1j}\frac{G({\bf a}_2)_j}{g_j}=\lambda\frac{G({\bf a}_2)_j^2}{g_j^2}<0,$$ 
and this implies the desired conclusion. Here $(c_B)_1,(c_B)_2$ represent the first two coordinates of the vector ${\bf c}_B$. The converse follows immediately with a similar argument.  \qed

\end{proof}
It follows from Lemma~\ref{mixed_bouquet} and the definition of the vector ${\bf c}_B$ that if the non-free bouquet $B$ is non-mixed, then all nonzero coordinates of ${\bf c}_B$ are positive. 

The following vector now encodes the set of dependencies from the Gale transform, and, therefore, also all of the the essential bouquet information about $B$.  
 
\begin{Definition} \rm  \label{defn:a_b}
Let $B$ be a bouquet of $A$ and define 
\[ 
  \ab_B:=\sum_{i=1}^n (c_B)_i\ab_i,
\] 
where $(c_B)_i$ denotes the $i$-th coordinate of the vector ${\bf c}_B$. The set of all vectors ${\bf a}_B$ corresponding to all bouquets of $A$ will be denoted by $A_B$, and thought of as a matrix with columns $\ab_B$. 
\end{Definition}

Let us isolate two special cases. First, if $B$ consists of just an isolated vertex $\ab_i$, then, when $\ab_i$ is not free ${\bf a}_B={\bf a}_i$, while otherwise $\ab_B$ can be any  positive multiple of $\ab_i$. Second, if the bouquet graph of $A$ has no mixed bouquets, then all vectors ${\bf a}_B$ corresponding to the non-free bouquets $B$ are just positive linear combinations of the vectors ${\bf a}_i\in B$. Finally, note that passing from $A$ to $A_B$ might not affect the matrix, since it might happen that $A=A_B$; a trivial example is $A_B=(A_B)_B$.
  
\begin{Remark}\label{rmk:FreeBouqAndZeros}
\rm If $A$ has a free bouquet, say $B_s$, then it is easy to see that 
$$\Ker_{\ZZ}(A_B)=\{(u_1,\ldots,u_{s-1},0): \ (u_1,\ldots,u_{s-1})\in\Ker_{\ZZ}(A_{B'})\},$$
where $A_{B'}=[\ab_{B_1},\ldots,\ab_{B_{s-1}}]$. In particular, we notice that for a free bouquet $B$, even though there are infinitely many choices to define $\cb_B$ and thus implicitly $\ab_B$, $\Ker_{\ZZ}(A_B)$ and $\Ker_{\ZZ}(A)$ are independent of that choice.  
\end{Remark} 

The following construction is the key result of this section, providing a bijection between the kernels of the toric matrix $A$ and the bouquet matrix $A_B$ from Definition~\ref{defn:a_b}.

\begin{Theorem}\label{gen_iso_kernels}
Suppose the bouquets of $A=[{\bf a}_1,\dots,{\bf a}_n]$  are $B_1, \dots, B_s$, and define $A_B=[\ab_{B_1},\dots,\ab_{B_s}]$. 
There is a bijective correspondence between the elements of $\Ker_{\ZZ}(A)$ and the elements of $\Ker_{\ZZ}(A_B)$ given by the map $\ub\mapsto B(\ub)$, where for  ${\bf u}=(u_1,\ldots,u_{s})\in\Ker_{\ZZ}(A_B)$  $$B({\bf u}):={\bf c}_{B_1}u_1+\cdots+{\bf c}_{B_s}u_s.$$ 
\end{Theorem}
\begin{proof}
 We may assume that $A$ has no free  bouquet by Remark~\ref{rmk:FreeBouqAndZeros}. 
 First we show that for every ${\bf v}=(v_1,\ldots,v_n)\in\Ker_{\ZZ}(A)$ there exists a vector ${\bf u}\in \Ker_{\ZZ}(A_B)$ such that ${\bf v}=B({\bf u})$. In order to prove this we may assume, without loss of generality, that there exist integers $1\leq i_1<i_2<\cdots<i_{s-1}\leq n$ such that ${\bf a}_1,\ldots,{\bf a}_{i_1}$ belong to the bouquet $B_1$, and so on, until ${\bf a}_{i_{s-1}+1},\ldots,{\bf a}_n$ belong to the bouquet $B_s$. Since ${\bf v}\in\Ker_{\ZZ}(A)$ it follows that there exist integers $\lambda_1,\ldots,\lambda_{n-r}$ such that ${\bf v}=\lambda_1 G_1+\cdots+\lambda_{n-r} G_{n-r}$. In other words, for each $k=1,\ldots,n$,  
$$v_k=\lambda_1 G({\bf a}_k)_1+\cdots+\lambda_{n-r}G({\bf a}_k)_{n-r}.$$

From the formula of the vectors ${\bf c}_{B_k}$ it follows that  $(c_{B_1})_1=\varepsilon_{1j} G({\bf a_1})_j/g_j$ for all nonzero $g_j=\gcd(G({\bf a}_1)_j,\ldots,G({\bf a}_{i_1})_j)$. Therefore, $G({\bf a_1})_j=(c_{B_1})_1\varepsilon_{1j}g_j$ for all $j=1,\ldots,n-r$, if we put $g_j=0$ when $G({\bf a}_1)_j=0$. Then we obtain $v_1=\sum_{j=1}^{n-r} \lambda_j G({\bf a}_1)_j=\sum_{j=1}^{n-r} \lambda_j\varepsilon_{1j}g_j(c_{B_1})_1$. Since ${\bf a}_1,\ldots,{\bf a}_{i_1}$ are in the same bouquet $B_1$, for any $k=1,\ldots,i_1$, it follows similarly that $G({\bf a}_k)_j=\varepsilon_{1j}g_j (c_{B_1})_k$ and  
$$v_k=\sum_{j=1}^{n-r}\lambda_jG({\bf a}_k)_j=\sum_{j=1}^{n-r}\lambda_j \varepsilon_{1j}g_j (c_{B_1})_k=u_1(c_{B_1})_k,$$
where $u_1=\sum_{j=1}^{n-r}\lambda_j \varepsilon_{1j}g_j$. This takes care of the first $i_1$ coordinates of ${\bf v}$, which correspond to the same bouquet $B_1$. The  argument for the remaining coordinates follows similarly.

Next we show that ${\bf u}\in\Ker_{\ZZ}(A_B)$ if and only if $B({\bf u})\in\Ker_{\ZZ}(A)$. Indeed, ${\bf u}=(u_1,\ldots,u_s)\in\Ker_{\ZZ}(A_B)$
is equivalent to $\sum_{k=1}^s u_k{\bf a}_{B_k}=0$. Replacing every vector ${\bf a}_{B_k}$ provides the following equivalent statement:  ${\bf u}\in\Ker_{\ZZ}(A_B)$ if and only if 
\[
0=\sum_{k=1}^s u_k{\bf a}_{B_k}=\sum_{k=1}^s u_k(\sum_{i=1}^n (c_{B_k})_i{\bf a}_i)=\sum_{i=1}^n (\sum_{k=1}^s u_k(c_{B_k})_i){\bf a}_i.
\] 
The latter sum being equal to zero is equivalent to 
\[
(\sum_{k=1}^s (c_{B_k})_1u_k,\sum_{k=1}^s (c_{B_k})_2u_k,\ldots,\sum_{k=1}^s (c_{B_k})_nu_k)\in\Ker_{\ZZ}(A), 
\]
or, more concisely, 
\[
\sum_{k=1}^s {\bf c}_{B_k}u_k=B({\bf u})\in\Ker_{\ZZ}(A),
\]
and the claim follows. \qed

\end{proof}

It is an immediate consequence of Theorem~\ref{gen_iso_kernels} that the ideals $I_A$ and $I_{A_B}$ have the same codimension, as the kernels of the two matrices have the same rank. 

\begin{Example}[Example~\ref{bouq_construction}, continued]
\label{c_vector}
{\em Let $A$ be the matrix from Example~\ref{bouq_construction}. It has three bouquets, $B_1,B_2,B_3$, with two vertices each: $\{{\bf a}_1,{\bf a}_6\}$, $\{{\bf a}_2,{\bf a}_5\}$ and  $\{{\bf a}_3,{\bf a}_4\}$, and  the isolated vertex $\ab_7$ as the free bouquet $B_4$. 
 Let us compute the  nonzero vectors ${\bf c}_{B_1},{\bf c}_{B_2},{\bf c}_{B_3}, \cb_{B_4}\in\ZZ^7$. For ${\bf c}_{B_1}$,  $j$ can be chosen either $1$ or $2$, while  $i_0=1$. Fix $j=1$. Thus $g_1=\gcd(G({\bf a}_1)_1,G({\bf a}_6)_1)=1$ and the nonzero coordinates of ${\bf c}_{B_1}$ are
\[
(c_{B_1})_1= \varepsilon_{11} \frac{G({\bf a}_1)_1}{g_1}=1, \ (c_{B_1})_6= \varepsilon_{11} \frac{G({\bf a}_6)_1}{g_1}=1. 
\]      
Hence ${\bf c}_{B_1}=(1,0,0,0,0,1,0)$ and ${\bf a}_{B_1}={\bf a}_1+{\bf a}_6=(1,1,1,1,0)$. Similarly it can be computed that ${\bf c}_{B_2}=(0,1,0,0,1,0,0)$, ${\bf a}_{B_2}={\bf a}_2+{\bf a}_5=(1,1,1,1,0)$, ${\bf c}_{B_3}=(0,0,1,1,0,0,0)$ and ${\bf a}_{B_3}={\bf a}_3+{\bf a}_4=(1,1,1,1,0)$,  while by definition  $\cb_{B_4}=(0,0,0,0,0,0,1)$ and $\ab_{B_4}=\ab_7=(0,0,0,0,3)$. Then $A_B$ consists of four vectors ${\bf a}_{B_1},{\bf a}_{B_2},{\bf a}_{B_3},{\ab_{B_4}}$, and consequently $\Ker_{\ZZ}(A_B)=\{(\alpha+\beta,-\alpha,-\beta,0)| \ \alpha,\beta\in\ZZ\}$. 
This encodes the vector
 $$B((\alpha+\beta,-\alpha,-\beta,0))=(\alpha+\beta){\bf c}_{B_1}-\alpha{\bf c}_{B_2}-\beta{\bf c}_{B_3}=(\alpha+\beta,-\alpha,-\beta,-\beta,-\alpha,\alpha+\beta,0),$$    
which is a generic element of $\Ker_{\ZZ}(A)$.} 
\end{Example}

Let $\bf u$, ${\bf w}_1,{\bf w}_2\in\Ker_{\ZZ}(A)$ be such that ${\bf u}={\bf w}_1 + {\bf w}_2$. Such a sum is a {\em conformal decomposition of}  ${\bf u}$, written as ${\bf u}={\bf w}_1 +_{c} {\bf w}_2$, if  ${\bf u}^+={\bf w}_1^+ + {\bf w}_2^+$ and ${\bf u}^-={\bf w}_1^- + {\bf w}_2^-$, see \cite{St}. If both ${\bf w}_1,{\bf w}_2$ are nonzero, the  decomposition is said to be {\em proper}.  
Recall the  equivalent definition of the Graver basis $\Gr(A)$ as  the (finite) set of nonzero vectors in $\Ker_{\ZZ}(A)$ for which there is no proper conformal decomposition (see, for example, \cite[Algorithm 7.2]{St}). Also, a vector ${\bf u}\in \Ker_{\ZZ}(A)$ is called a {\em circuit} of $A$ if  supp$({\bf u})$ is minimal with respect to inclusion and the coordinates of ${\bf u}$ are relatively prime. As usual, we denote by ${\mathcal C}(A)$ the set of circuits of $A$.

\begin{Theorem}\label{gen_all_is_well}
Let $A$ and $A_B$ be as in  Theorem~\ref{gen_iso_kernels}. Then there is a bijective correspondence between the Graver basis of $A_B$ and the Graver basis of $A$, and a similar bijective correspondence holds between the sets of circuits. Explicitly: 
\[ \Gr(A)=\{B({\bf u})| \ {\bf u}\in\Gr(A_B)\}\ \
 \text{and} \ \ {\mathcal C}(A)=\{B({\bf u})| \ {\bf u}\in{\mathcal C}(A_B)\}. \]
\end{Theorem}
\begin{proof}
Without loss of generality we may assume, as before, that $A$ has $s$ non-free bouquets and there exist integers $1\leq i_1<i_2<\cdots<i_{s-1}\leq n$ such that ${\bf a}_1,\ldots,{\bf a}_{i_1}$ belong to the bouquet $B_1$, and so on, unti ${\bf a}_{i_{s-1}+1},\ldots,{\bf a}_n$ belong to the bouquet $B_s$. It follows from Theorem~\ref{gen_iso_kernels} that the Graver basis elements of $A$ are of the form $B({\bf u})$ with ${\bf u}\in\Ker_{\ZZ}(A_B)$. Moreover, for any ${\bf u}=(u_1,\ldots,u_s)\in\Ker_{\ZZ}(A_B)$, Theorem~\ref{gen_iso_kernels} implies that 
\begin{eqnarray}\label{all}
B({\bf u})=((c_{B_1})_1u_1,\ldots,(c_{B_1})_{i_1}u_1,\ldots,(c_{B_s})_{i_{s-1}+1}u_s,\ldots,(c_{B_s})_nu_s).
\end{eqnarray}
The bijective correspondence between the Graver basis elements follows at once since for any ${\bf u},{\bf v},{\bf w}\in\Ker_{\ZZ}(A_B)$,  we have ${\bf u}={\bf v}+_{c}{\bf w}$ if and only if $B({\bf u})=B({\bf v})+_{c}B({\bf w})$. 
Using Equation~\eqref{all} 
again provides that $\Supp(B({\bf u}))=\cup_{i\in\Supp({\bf u})}\Supp({\bf c}_{B_i})$. Thus, ${\bf u}$ is a circuit if and only if $B({\bf u})$ is a circuit.  \qed

\end{proof}

This correspondence in general does not offer any relation between Markov bases, indispensable binomials and universal Gr\"obner bases of $A$ and $A_B$. On the other hand, Section~\ref{sec:OnStableToricIdeals} shows that the correspondence preserves additional sets for the case of stable toric ideals (see Definition~\ref{stable_1}). 

\begin{Example}[Example~\ref{c_vector}, continued]
\rm
One easily checks that  $\Gr(A_B)=\mathcal{C}(A_B)=\{(1,-1,0,0),(0,1,-1,0),(1,0,-1,0)\}$, while $\Gr(A)=\mathcal{C}(A)$ and consists of the vectors $(1,-1,0,0,-1,1,0),(0,1,-1,-1,1,0,0),(1,0,-1,-1,0,1,0)$ , since $$B((u_1, u_2, u_3, u_4))=u_1\cb_{B_1}+u_2\cb_{B_2}+u_3\cb_{B_3}+u_4\cb_{B_4}=(u_1, u_2, u_3, u_3, u_2, u_1, u_4).$$ The one-to-one correspondence for the two sets holds in the order indicated. For example, 
\[
	B((0,1,-1,0))={\bf c}_{B_2}-{\bf c}_{B_3}=(0,1,-1,-1,1,0,0).
\] 
\end{Example}

\begin{Remark}\label{subbouquets}
{\em  A {\em subbouquet} of $G_A$ is an induced subgraph of $G_A$, which is a clique on its set of vertices. Note that a bouquet is a maximal subbouquet. We will say that $A$ has a {\em subbouquet decomposition} if there exists a family of subbouquets, say $B_1,\ldots,B_t$, such that they are pairwise vertex disjoint and their union of vertices equals $\{\ab_1,\ldots,\ab_n\}$. A subbouquet decomposition always exists if we consider, for example, the subbouquet decomposition induced by all of the bouquets. In addition, Theorem~\ref{gen_iso_kernels} and Theorem \ref{gen_all_is_well} are true even if we replace bouquets with proper subbouquets which form a subbouquet decomposition of $A$. The following sections utilize this when applying the two results.}
\end{Remark}

 We conclude this section with a rather surprising application of the bouquet construction. It will turn out that passing from a unimodular matrix $A$ to its (sub)bouquet matrix $A_B$ preserves unimodularity. To this end, recall that a matrix $A$ is called {\em unimodular} if all of its nonzero maximal minors have the same absolute value, see for example \cite[Section 8, page 70]{St}. A nice property of a unimodular matrix $A$ is that for its toric ideal $I_A$ the set of circuits equals the Graver basis (and thus also equals the universal Gr\"obner basis), and all coordinates of the Graver basis of $A$ belong to $\{0,-1,+1\}$, see for example \cite[Proposition 8.11]{St}.  

\begin{Proposition}\label{unimodular_ok}    
Suppose that $A=[\ab_1,\ldots,\ab_n]$ has the subbouquet decomposition $B_1,\ldots,B_s$ and let $A_B=[\ab_{B_1},\ldots,\ab_{B_s}]$. Then $A$ is unimodular if and only if $A_B$ is unimodular and all of the nonzero coordinates of the vectors $\cb_{B_1},\ldots,\cb_{B_s}$ are $+1$ or $-1$.
\end{Proposition}
\begin{proof}
Assume first that $A$ is unimodular. Then it follows from Theorem~\ref{gen_all_is_well} and \cite[Proposition 8.11]{St} that the set of circuits of $A_B$ equals the Graver basis of $A_B$. It is a known fact that a matrix is unimodular if and only if all initial ideals of its toric ideal are squarefree, see \cite[Remark 8.10]{St}. Thus, for proving the unimodularity of $A_B$, it is sufficient to show that the Graver basis of $A_B$, hence also the universal Gr\"obner basis, consists of vectors with nonzero coordinates only $-1$ or $+1$. Let ${\bf u}=(u_1,\ldots,u_s)$ be a Graver basis element of $A_B$. Then $B({\bf u})=\sum_{i=1}^s \cb_{B_i}u_i$ is a Graver basis element of $A$ and, since $A$ is unimodular, all coordinates of $B({\bf u})$ are $0,-1,1$. Because the vectors $\cb_{B_i}$ have the supports pairwise disjoint, the latter condition is fulfilled only if for all $i$ we have $u_i\in\{0,-1,1\}$ and the nonzero coordinates of $\cb_{B_i}$ are $1$ or $-1$. Thus we obtain the desired conclusion. For the converse, by Theorem~\ref{gen_all_is_well} the Graver basis elements of $A$ are of the form $B({\bf u})=\sum_{i=1}^s \cb_{B_i}u_i$ with ${\bf u}$ running over all Graver basis elements of $A_B$. Since $A_B$ is unimodular then ${\bf u}$ has all nonzero coordinates $+1$ or $-1$ and thus, by the hypothesis on $\cb_{B_i}$, it follows that $B({\bf u})$ has all nonzero coordinates $+1$ or $-1$. The unimodularity of $A_B$ implies via Theorem~\ref{gen_all_is_well} the equality of the set of circuits of $A$ and the Graver basis of $A$, and implicitly the equality with the universal Gr\"obner basis of $A$. Therefore the universal Gr\"obner basis of $A$ consists also of vectors with nonzero coordinates either $1$ or $-1$, and this implies that $A$ is unimodular.\qed  
     
\end{proof}

\section{ Generalized Lawrence matrices }
\label{sec:GeneralizedLawrenceMatrices}

 This section is dedicated to the construction of a natural inverse procedure of the one given in Section~\ref{sec:BouqDec}. Namely, given an arbitrary set of vectors  $\ab_1,\ldots,\ab_s$ and vectors $\cb_1,\ldots,\cb_s$ that can act as bouquet-index-encoding vectors (cf.\ definition of $\cb_B$ in Section~\ref{sec:BouqDec}), the following result constructs  a toric ideal $I_A$ whose $s$ subbouquets are encoded by the given vectors. 

Recall that an integral vector $\ab\in\ZZ^m$ is primitive if the greatest common divisor of all its coordinates is $1$.  

\begin{Theorem}\label{inverse_construction}
	Let $\{\ab_1,\ldots,\ab_s\}\subset\ZZ^m$ be an arbitrary set of vectors. Let $\cb_1,\ldots,\cb_s$ be any set of primitive vectors, with $\cb_i\in\ZZ^{m_i}$ for some $m_i\geq 1$, each having full support and a positive first coordinate. Define $p=m+\sum_{i=1}^s (m_i-1)$ and $q=\sum_{i=1}^s m_i$. Then, there exists a matrix  $A\in\ZZ^{p\times q}$ with a subbouquet decomposition, $B_1,\ldots,B_s$, such that the $i^{th}$ subbouquet	is encoded by  the following vectors: $\ab_{B_i}=(\ab_i,{\bf 0},\ldots,{\bf 0})\in\ZZ^{p}$ and $\cb_{B_i}=({\bf 0},\ldots,\cb_i,\ldots,{\bf 0})\in\ZZ^{q}$, where the support of $\cb_{B_i}$ is precisely in the $i^{th}$ block of $\ZZ^q$ of size $m_i$. 
\end{Theorem}
\begin{proof}
We begin by constructing $A$ explicitly, and then show its subbouquets satisfy the required conditions.	For each $i=1,\dots,s$, let  $\cb_i=(c_{i1},\ldots,c_{im_i})\in\ZZ^{m_i}$ and define 
\[
C_i=
\left( \begin{array}{ccccc}
 -c_{i2} &  c_{i1} &   &  &\  \\
 -c_{i3} &  & c_{i1} &  &\  \\
& & &  \ddots &  \\
 -c_{im_i} &  &  &   &\ c_{i1}
\end{array} \right)
	 \in\ZZ^{(m_i-1)\times m_i}.
\]
	Primitivity of each $\cb_i$ implies that there exist integers $\lambda_{i1},\ldots,\lambda_{im_i}$ such that $1=\lambda_{i1}c_{i1}+\cdots+\lambda_{im_i}c_{im_i}$. Fix a choice of $\lambda_{i1},\ldots,\lambda_{im_i}$, and define the matrices $A_i=[\lambda_{i1}{\bf a}_i,\ldots,\lambda_{im_i}{\bf a}_i]\in\ZZ^{m\times m_i}$. 	The desired matrix $A$ is then the following block matrix: 
\[
A=\
\left( \begin{array}{cccc}
 A_1 & \ A_2 & \cdots & \ A_s \\
 C_1 &\ 0\ & \cdots & 0 \\
 0 &\ C_2\ & \cdots & 0 \\
\vdots&\vdots & \ddots & \vdots \\
 0 &\ 0 & \cdots &\ C_s
\end{array} \right)  \in \ZZ^{p\times q}, 
\]
where $p=m+(m_1-1)+\cdots+(m_s-1)$ and $q=m_1+\cdots+m_s$.	We will  denote the columns of the matrix $A$ by $\beta_1,\ldots,\beta_q$. 
	
	In the remainder of the proof, we show that $\beta_1,\ldots,\beta_{m_1}$ belong to the same subbouquet $B_1$. Analogous arguments, with a straightforward shift of the indices, apply to show that $\beta_{m_1+\cdots+m_{i-1}+1},\ldots,\beta_{m_1+\cdots+m_i}$ belong to the same subbouquet $B_i$ for all $i=2,\ldots,s$.  
	
	Consider the vector $\gamma_{m+i}\in\ZZ^p$, whose only nonzero coordinate is 1, in position $m+i-1$, for every i=$2,\ldots,m_1$. Then the co-vector  $$(\gamma_{m+i}\cdot\beta_1,\ldots,\gamma_{m+i}\cdot\beta_{q})=(-c_{1i},0,\ldots,c_{11},0,\ldots,0)$$ has support $\{1,i\}$, and by Remark~\ref{cocircuit_gale}, the following relation holds for any $2\leq i\leq m_1$:  
\[
	\tag{*}\label{*} -c_{1i}G(\beta_1)+c_{11}G(\beta_i)={\bf 0}. 
\]

Since all coordinates of ${\bf c}_1$ are nonzero, the previous relations imply that $\beta_1,\ldots,\beta_{m_1}$ belong to the same subbouquet,  which may or may not be free. Therefore $A$ has $s$ subbouquets.

	 Finally, we compute $\cb_{B_1}$. If  $B_1$ is free, then by definition we can take $\cb_{B_1}=(\cb_1,{\bf 0},\ldots,{\bf 0})$. Otherwise, if $B_1$ is not
	 free, then there exists a coordinate $j$ such that $G(\beta_i)_j\neq 0$ for all $i=1,\ldots,m_1$. Then the relations \eqref{*} provide
\[
\frac{G(\beta_1)_j}{c_{11}}=\ldots =\frac{G(\beta_{m_1})_j}{c_{1m_1}}=\frac{k}{l},
\]  
with relatively prime integers $k$ and $l$, $l>0$. Thus $lG(\beta_i)_j=c_{1i}k$, and consequently $l|c_{1i}$ for all $i$. But the coordinates of $c_1$ being relatively prime implies $l=1$ and thus $G(\beta_i)_j=c_{1i}k$ for all $i$. Therefore $g_j:=gcd(G(\beta_1)_j,\ldots,G(\beta_{m_1})_j)=|k|$. On the other hand, since $G(\beta_1)_j=c_{11}k$ and $c_{11}>0$ it follows that $\varepsilon_{1j}=\sgn(k)$, where $\varepsilon_{ij}$ is the sign of $G(\beta_i)_j$. Hence $G(\beta_i)_j=c_{1i}k=\varepsilon_{1j}c_{1i}g_j$, and by definition of $\cb_{B_1}$ we have $$(c_{B_1})_i=\varepsilon_{1j}\frac{G(\beta_i)_j}{g_{j}}=\varepsilon_{1j}\varepsilon_{1j}c_{1i}=c_{1i}, \ \text{ for all } i=1,\ldots,m_1.$$  
Thus $\cb_{B_1}=(\cb_1,{\bf 0},\ldots,{\bf 0})$, and $$\ab_{B_1}=\sum_{i=1}^{m_1}(c_{B_1})_i\beta_i=\sum_{i=1}^{m_1}c_{1i}\beta_i=(\ab_1,{\bf 0},\ldots,{\bf 0}),$$
as desired. Note that for the last equality we used that $\lambda_{11}c_{11}+\cdots+\lambda_{1m_1}c_{1m_1}=1$. \qed    

\end{proof}

\begin{Remark}\label{why_gen_lawrence}
{\em  We will call the matrix $A$ defined in Theorem~\ref{inverse_construction} the {\em generalized Lawrence matrix}, since in particular one can recover, after a column permutation, the classical second Lawrence lifting. Indeed, recall from \cite[Chapter 7]{St} that the second Lawrence lifting $\Lambda(D)$ of an integer matrix $D=[{\bf d}_1,\ldots,{\bf d}_n]\in\ZZ^{m\times n}$ is defined as
\[
\Lambda(D)=
\left(\begin{array}{cc}
 D &  0 \\
 I_n & I_n \\
\end{array}\right)\in\ZZ^{ (m+n)\times 2n}.
\]
Applying the construction of Theorem~\ref{inverse_construction} for the vectors $\ab_i={\bf d}_i$ and $\cb_i=(1,-1)$ for all $i=1,\ldots,n$, with $\lambda_{i1}=1$ and $\lambda_{i2}=0$, we obtain
\[
A_i=
\left(\begin{array}{cc}
 {\bf d}_i &  {\bf 0} \\
\end{array}\right)\in\ZZ^{m\times 2}, \quad C_i=\left(\begin{array}{cc}
 1 &  1 \\
\end{array}\right)\in\ZZ^{1\times 2},
\]
and thus $A$ is just $\Lambda(D)$ after a column permutation. }
\end{Remark}

As an immediate consequence of Theorem~\ref{inverse_construction} we obtain that generalized Lawrence matrices capture the ideal information for any general matrix $A$:

\begin{Corollary}\label{all_toric_is_gen_lawrence} 
  For any integer matrix $A$ there exists a generalized Lawrence matrix $A'$ such that $I_A=I_{A'}$, up to permutation of column indices.
\end{Corollary}  
\begin{proof} 
Let $A=[\ab_1,\ldots,\ab_n]\in\ZZ^{m\times n}$, and assume that $A$ has $s$ bouquets $B_1,\ldots,B_s$. Then $A_B=[\ab_{B_1},\ldots,\ab_{B_s}]\in\ZZ^{m\times s}$,  
and $\cb_{B_1},\ldots,\cb_{B_s}\in\ZZ^n$  are such that their supports are pairwise disjoint and $\cup_{i=1}^s \supp(\cb_{B_i})=[n]$. After permuting the columns of $A$ we may assume that $\cb_{B_i}=({\bf 0},\ldots,\cb_i,\ldots,{\bf 0})$ for some $\cb_i\in\ZZ^{m_i}$ with full support (that is, $m_i=|\supp{\cb_{B_i}}|$) and primitive.  Theorem~\ref{inverse_construction} applied to the set of vectors $\{\ab_{B_1},\ldots,\ab_{B_s}\}$ and $\{\cb_1,\ldots,\cb_s\}$ provides the existence of a matrix $A'\in\ZZ^{p\times n}$ with $s$ subbouquets $B'_1,\ldots,B'_s$ and with the property that $\ab_{B'_i}=(\ab_{B_i},{\bf 0},\ldots,{\bf 0})$  and $\cb_{B'_i}=({\bf 0},\ldots,\cb_i,\ldots,{\bf 0})$ for all $i$.  Since  $\cb_{B'_i}=\cb_{B_i}$ for all $i=1,\ldots,s$,  Theorem~\ref{gen_iso_kernels} implies that $\Ker_{\ZZ}(A)=\Ker_{\ZZ}(A')$. \qed

\end{proof}

 The following is  an example  of how to apply Theorem~\ref{inverse_construction}.
\begin{Example}\label{how_inverse_proc_goes}
{\em Let $\cb_1=(2,23)$, $\cb_2=(3,-2,-4,2017)$ and $\cb_3=(11,-5,8)$ be three primitive vectors and $\ab_1=3$, $\ab_2=4$, $\ab_3=5$. Following the construction from the proof of Theorem~\ref{inverse_construction} we obtain that the matrix 
\[
A=[\ab'_1,\ldots,\ab'_9]=\
\left( \begin{array}{ccccccccc}
 105 & -9 & 36 & 20 & 16 & 0 & 30 & 25 & -25 \\
 -23 & 2 & 0 & 0 & 0 & 0 & 0 & 0 & 0 \\
 0 & 0 & 2 & 3 & 0 & 0 & 0 & 0 & 0 \\
 0 & 0 & 4 & 0 & 3 & 0 & 0 & 0 & 0 \\
 0 & 0 & -2017 & 0 & 0 & 3 & 0 & 0 & 0 \\
 0 & 0 & 0 & 0 & 0 & 0 & 5 & 11 & 0 \\
 0 & 0 & 0 & 0 & 0 & 0 & -8 & 0 & 11  
\end{array} \right)  \in \ZZ^{7\times 9}, 
\] 
for the choices of $\lambda_{11}=35, \lambda_{12}=-3$, $\lambda_{21}=9, \lambda_{22}=5, \lambda_{23}=4, \lambda_{24}=0$, and $\lambda_{31}=6, \lambda_{32}=5, \lambda_{33}=-5$, respectively. The matrix $A$ has a subbouquet decomposition induced by the three subbouquets $B_1,B_2,B_3${\blue , }  with the first one consisting of $\ab'_1,\ab'_2$, the second one consisting of $\ab'_3,\ab'_4,\ab'_5,\ab'_6${\blue , }  and the last one consisting of $\ab'_7,\ab'_8,\ab'_9$. The subbouquet ideal $I_{A_B}$ is just the toric ideal of the matrix $A_B=[\ab_{B_1},\ab_{B_2},\ab_{B_3}]$ with $\ab_{B_1}=(3,0,\ldots,0)\in\ZZ^7$, $\ab_{B_2}=(4,0,\ldots,0)\in\ZZ^7$ and $\ab_{B_3}=(5,0,\ldots,0)\in\ZZ^7$. }
\end{Example}

The most important property of Theorem~\ref{inverse_construction}, often exploited in the next sections, is that it can be used to provide infinite classes of examples. For example, if we want to construct infinitely many unimodular matrices we proceed as follows: let $D=[\ab_1,\ldots,\ab_s]$ be any unimodular matrix (for example the incidence matrix of a bipartite graph) and, based on Proposition~\ref{unimodular_ok} choose arbitrary $\cb_1,\ldots,\cb_s$ with entries $-1$ or $1$, and the first nonzero coordinate being $1$. Then Theorem~\ref{inverse_construction} yields the generalized Lawrence matrix $A$ such that its subbouquet ideal equals $I_D$, and finally by Proposition~\ref{unimodular_ok} we obtain that $A$ is also unimodular. Similarly, from an arbitrary unimodular matrix $D$, using any set of vectors $\cb_1,\ldots,\cb_s$ satisfying the hypotheses of Theorem~\ref{inverse_construction} and such that at least one has one coordinate in absolute value larger than 1, we can construct infinitely many generalized Lawrence matrices that are not unimodular, but have the set of circuits equal to the universal Gr\"obner basis and the Graver basis.

\begin{Remark}\label{inverse_graph}
\rm Theorem~\ref{inverse_construction} also solves the following two natural problems. First, given an arbitrary graph $G$ whose connected components are cliques, there exists a matrix $A$ such that the bouquet graph $G_A$ of $A$ is precisely $G$. Second, a stronger statement can be made: given a graph $G$ whose connected components are cliques, along with $+$ and $-$ signs associated to each edge of $G$ according to the sign rules explained following Remark~\ref{cocircuit_gale}, there exists a matrix $A$ whose bouquet graph and structure are encoded by $G$.
\end{Remark}

\section{On stable toric ideals}
\label{sec:OnStableToricIdeals}

In this section we prove that a certain bouquet structure of $A$, which we call \emph{stability}, provides many additional properties of the map ${\bf u}\mapsto B({\bf u})$ between $\Ker_{\ZZ}(A_B)$ and $\Ker_{\ZZ}(A)$. We will assume, throughout this section, that $A$ has no free bouquet. If $A$ has free bouquets, then all of the results in this section remain valid, since  free bouquets do not affect the kernels $\Ker_{\ZZ}(A)$ and $\Ker_{\ZZ}(A_B)$ by Remark~\ref{rmk:FreeBouqAndZeros} and Theorem~\ref{gen_iso_kernels}.

\begin{Definition}\label{stable_1}
{\em The toric ideal $I_A$ is called {\em stable} if all of the bouquets of $A$ are non-mixed. More generaly, the toric ideal $I_A$ is called {\em stable with respect to a subbouquet decomposition} of $A$ if there exists a subbouquet decomposition of $A$, such that all of the subbouquets are non-mixed.}
\end{Definition}

Note that there always exists a subbouquet decomposition such that $I_A$ is stable with respect to it: the trivial subbouquet decomposition, i.e. all subbouquets are isolated vertices. In general, there might be several different such subbouquet decompositions. However, there is a canonical (and maximal) subbouquet decomposition such that $I_A$ is stable with respect to it, see Remark~\ref{first_stable}. Still, $I_A$ can be stable with respect to a subbouquet decomposition, without being stable, see Example~\ref{general_decomposition}. On the other hand, a stable ideal $I_A$ is obviously stable with respect to the subbouquet decomposition given by the family of bouquets of $A$.
 
In the case of stable toric ideals the map ${\bf u}\mapsto B({\bf u})$ has the following additional property.   

\begin{Remark}\label{B(u)}
{\em If all of the bouquets of $A$ are non-mixed, then it follows from Lemma~\ref{mixed_bouquet} that the vectors $\cb_{B_1},\ldots,\cb_{B_s}$ have all nonzero coordinates positive. Then, by Theorem~\ref{gen_iso_kernels}, it follows that $B({\bf u})^+=B({\bf u}^+)$ and $B({\bf u})^-=B({\bf u}^-)$ for every ${\bf u}\in\Ker_{\ZZ}(A_B)$.}
\end{Remark}

Stability ensures that several of the properties of $I_A$ are preserved when passing to $I_{A_B}$. We begin with an easy, but crucial, Lemma.

\begin{Lemma}\label{mixed_same_positive}
Let $I_A$ be a stable toric ideal. Then $I_A$ is positively graded if and only if $I_{A_B}$ is positively graded.
\end{Lemma}
\begin{proof}
  Let  ${\bf v}\in\Ker_{\ZZ}(A)$. By Theorem~\ref{gen_iso_kernels} there exists a vector ${\bf u}$ such that ${\bf v}=B({\bf u})$ for some ${\bf u}\in\Ker_{\ZZ}(A_B)$. Therefore ${\bf v}=\sum_{i=1}^s {\cb_{B_i}}u_i$, and since every $\cb_{B_i}$ is a nonzero vector with all nonzero coordinates positive, and their supports are pairwise disjoint, then we obtain that ${\bf 0}\neq{\bf v}\in\NN^n$ if and only if ${\bf 0}\neq {\bf u}\in\NN^s$.  \qed

\end{proof}

 A {\em Markov basis} of $A$ is a finite subset $\MM$ of $\Ker_{\ZZ}(A)$ such that whenever ${\bf w}, {\bf v}\in \NN^{n}$ and ${\bf w}- {\bf v}\in \Ker_{\ZZ}(A)$  there exists a subset $\{{\bf u}_i: i=1,\ldots, r\}$ of  ${\MM}$ that {\it connects} ${\bf w}$ to ${\bf v}$.  Here, connectedness  means that  ${\bf w}- {\bf v}=\sum^r_{i=1}{\bf u}_i $, and $({\bf w}- {{\sum^p_{i=1}}}{{\bf u}_i})\in \NN^n$  for all $1\leq p\leq r$.
 
 A Markov basis $\MM$ is {\em minimal} if no proper subset of $\MM$ is a Markov basis. By a fundamental theorem from Markov bases literature (see \cite{DiSt}), a set of vectors is a Markov basis for $A$ if and only if the corresponding set of binomials (whose exponents are the given vectors) generate the toric ideal $I_A$.

\begin{Proposition}\label{stable_markov}
Let $I_A$ be a stable toric ideal. Then the map ${\bf u}\mapsto B({\bf u})$ is a bijective correspondence between the minimal Markov bases of $A_B$ and the minimal Markov bases of $A$.  
\end{Proposition}
\begin{proof}
Let $\MM$ be a Markov basis of $A_B$, and $\MM'=\{B({\bf u}): \ {\bf u}\in\MM\}$. Let ${\bf w}=B({\bf u})\in\Ker_{\ZZ}(A)$, for some ${\bf u}\in\Ker_{\ZZ}(A_B)$, see Theorem~\ref{gen_iso_kernels}. Then by Remark~\ref{B(u)} it follows that ${\bf w}^+=B({\bf u}^+)$ and ${\bf w}^-=B({\bf u}^-)$. Since ${\bf u}\in\Ker_{\ZZ}(A_B)$ and $\MM$ is a Markov basis for $A_B$ then there exists a subset $\{{\bf u}_i: i=1,\ldots,r\}$ of $\MM$ such that ${\bf u}^+-{\bf u}^-=\sum^r_{i=1}{\bf u}_i$ and $({\bf u}^+- {{\sum^p_{i=1}}}{{\bf u}_i})\in \NN^s$ for all $1\leq p\leq r$. Thus
\[
{\bf w}^+-{\bf w}^-=B({\bf u})=B(\sum^r_{i=1}{\bf u}_i)=\sum_{i=1}^r B({\bf u}_i), 
\]
and 
\[
{\bf w}^+-\sum_{i=1}^p B({\bf u}_i)=B({\bf u}^+-\sum_{i=1}^p{\bf u}_i)
\]
belongs to $\NN^n$ by Remark~\ref{B(u)}, since the vector ${\bf u}^+-\sum_{i=1}^p{\bf u}_i\in\NN^s$ for all $1\leq p\leq r$. Therefore $\MM'$ is a Markov basis of $A$. Conversely, given any Markov basis $\MM'$ of $A$ then its elements are of the form $\{B({\bf u}): {\bf u}\in\MM\subset\Ker_{\ZZ}(A_B)\}$, see Theorem~\ref{gen_iso_kernels}. As before, using Remark~\ref{B(u)} it can be shown that $\MM$ is a Markov basis of $A_B$. This bijective correspondence ensures also that the map ${\bf u}\mapsto B({\bf u})$ preserves the minimality of the Markov bases.  Indeed, if $\MM$ is a minimal Markov basis of $A_B$ then $\MM'=\{B({\bf u}): \ {\bf u}\in\MM\}$ is a Markov basis of $A$, and if it were not minimal then a proper subset $\MM_1'$ of it would be minimal, and thus we would obtain that a proper subset of $\MM$ would be a Markov basis of $A_B$, a contradiction. \qed

\end{proof}

The intersection of all (minimal) Markov bases of $A$ via the identification of the elements which differ by a sign is called the set of {\it indispensable  binomials} of $A$, and denoted by $\MS(A)$. In order to determine the indispensable binomials one has to deal with two cases: $\Ker_{\ZZ}(A)\cap\NN^n\neq\{\bf 0\}$ or $\Ker_{\ZZ}(A)\cap\NN^n=\{\bf 0\}$. In the first case, it follows from \cite[Theorem 4.18]{CTV1} that $\MS(A)=\emptyset$. As a side comment, note that if one restricts to the intersection of all minimal Markov bases of minimal cardinality then it may be at most one binomial in $\MS(A)$. In the second case,  \cite[Proposition 1.1]{CTV} provides the following useful algebraic characterization, which will be needed in Section~\ref{sec:EqualBases}: the set of indispensable binomials of $A$ consists of all binomials $x^{{\bf u}^+}-x^{{\bf u}^-}$ corresponding to the nonzero vectors ${\bf u}$ in $\Ker_{\ZZ}(A)$ which have no proper semiconformal decomposition. We recall from \cite[Definition 3.9]{HS} that for vectors ${\bf u},{\bf v},{\bf w}\in\Ker_{\ZZ}(A)$ such that ${\bf u}={\bf v}+{\bf w}$, the sum is said to be a  {\em semiconformal decomposition of} ${\bf u}$, written ${\bf u}={\bf v}+_{sc}{\bf w}$ , if $v_i>0$ implies that $w_i\geq 0$, and $w_i<0$ implies that $v_i\leq 0$ for all $1\leq i\leq n$. As before we call the decomposition proper if both ${\bf v},{\bf w}$ are nonzero.  Note that when writing a semiconformal decomposition of ${\bf u}$ it is necessary to specify the order of the vectors added.

\begin{Proposition}\label{stable_indispensable}
Let $I_A$ be a stable toric ideal. Then the map ${\bf u}\mapsto B({\bf u})$ induces a bijective correspondence between the indispensable binomials of $A_B$ and the indispensable binomials of $A$.  
\end{Proposition}
\begin{proof}
It follows from the above considerations that if $I_A$ is not positively graded then $A$ has no indispensable binomials. Thus by Lemma~\ref{mixed_same_positive}, we may assume that $I_A$ is positively graded, otherwise the conclusion holds trivially. Note that if ${\bf u}={\bf v}+_{sc}{\bf w}$ is a proper semiconformal decomposition of ${\bf u}$ then $B({\bf u})=B({\bf v})+_{sc}B({\bf w})$ is a proper semiconformal decomposition of $B({\bf u})$ and vice versa, since all $\cb_{B_i}$ are nonzero and with the nonzero coordinates positive. Thus, by applying \cite[Proposition 1.1]{CTV} we obtain the desired correspondence between the indispensable binomials. \qed

\end{proof}

A remark on the proof is in order. Since by definition $\MS(A)$ is the intersection of all (minimal) Markov bases via the identification of the elements which differ by a sign, then by Proposition~\ref{stable_markov} we could have obtained  the desired correspondence between $\MS(A)$ and $\MS(A_B)$. However, we prefer to prove it using semiconformal decompositions, as it provides a basis for some constructions in later sections.

\begin{Proposition}\label{stable_ugb}
Let $I_A$ be a stable toric ideal. Then the map ${\bf u}\mapsto B({\bf u})$ induces a bijective correspondence between the reduced Gr\"obner bases of $A_B$ and the reduced Gr\"obner bases of $A$. In particular, there is a bijective correspondence between the universal Gr\"obner bases of $A_B$ and $A$. 
\end{Proposition}
\begin{proof}
 As in the proof of Theorem~\ref{gen_iso_kernels} we may assume, for ease of notation, that there exist integers $1\leq i_1<i_2<\cdots<i_{s-1}\leq n$ such that ${\bf a}_1,\ldots,{\bf a}_{i_1}$ belong to the bouquet $B_1$, and so on, until ${\bf a}_{i_{s-1}+1},\ldots,{\bf a}_n$ belong to the bouquet $B_s$. Thus  $\supp(\cb_{B_1})=\{1,\ldots,i_1\}$, $\ldots$ , $\supp(\cb_{B_s})=\{i_{s-1}+1,\ldots,n\}$. 

Let $\mathcal G=\{y^{{\bf u_1}^+}-y^{{\bf u_1}^-},\ldots,y^{{\bf u_t}^+}-y^{{\bf u_t}^-}\}$ be a reduced Gr\"obner basis of $I_{A_B}$ with respect to a monomial order $<$ on $K[y_1,\ldots,y_s]$. By \cite[Proposition 1.11]{St} there exists a weight vector $\omega=(\omega_1,\ldots,\omega_s)\in\NN^s$ such that the monomial order $<$ is given by $\omega$, that is $\ini_{<}(I_{A_B})=\ini_{\omega}(I_{A_B})$. Without loss of generality we may assume that $y^{{\bf u_i}^+}>y^{{\bf u_i}^-}$, i.e. the dot product $\omega\cdot {\bf u_i}>0$, for any $i=1,\ldots,t$. We will define a monomial order $<_1$ on $K[x_1,\ldots,x_n]$ such that the set $\mathcal{G}'=\{x^{B({\bf u_1})^+}-x^{B({\bf u_1})^-},\ldots,x^{B({\bf u_t})^+}-x^{B({\bf u_t})^-}\}$ is a reduced Gr\"obner basis of $I_A$ with respect to $<_1$. For this, we first define the vector 
\[
\omega'=(\frac{\omega_1}{(c_{B_1})_{1}i_1},\ldots,\frac{\omega_1}{(c_{B_1})_{i_1}i_1},\frac{\omega_2}{(c_{B_2})_{i_1+1}(i_2-i_1)},  \ldots,\frac{\omega_{s}}{(c_{B_s})_{n}(n-i_{s-1})}),
\]
and let $\prec$ be an arbitrary monomial order on $K[x_1,\ldots,x_n]$. We define $<_1$ to be the monomial order $\prec_{\omega'}$ on $K[x_1,\ldots,x_n]$ induced by $\prec$ and $\omega'$, see \cite[Chapter 1]{St} for definition. Next we prove that $y^{{\bf u}^+}>y^{{\bf u}^-}$ implies that $x^{B({\bf u})^+}>_1 x^{B({\bf u})^-}$. Indeed, since $y^{{\bf u}^+}>y^{{\bf u}^-}$ then $\omega\cdot{\bf u}>0$, where ${\bf u}=(u_1,\ldots,u_s)$.  Thus
\[
\omega'\cdot B({\bf u})=\frac{\omega_1}{(c_{B_1})_{1}i_1}(c_{B_1})_1u_1+\frac{\omega_1}{(c_{B_1})_{2}i_1}(c_{B_1})_2u_1+\cdots=\omega\cdot{\bf u}>0, 
\]  
which implies that $x^{B({\bf u})^+}>_1 x^{B({\bf u})^-}$, as desired. In particular, we obtain that $x^{B({\bf u_i})^+}>_1 x^{B({\bf u_i})^-}$ for all $i$, and consequently $(x^{B({\bf u_1})^+}, \ldots,x^{B({\bf u_t})^+})\subset\ini_{<_1}(I_A)$. For the converse inclusion let $x^{B({\bf u})^+}- x^{B({\bf u})^-}\in I_A$ be an arbitrary element, see Theorem~\ref{gen_iso_kernels}, and say that $x^{B({\bf u})^+}>_1 x^{B({\bf u})^-}$, the other case being similar. It follows from the previous considerations that $y^{{\bf u}^+}>y^{{\bf u}^-}$ and thus there exists an integer $i$ such that $y^{{\bf u_i}^+}|y^{{\bf u}^+}$. Since ${\bf u}\mapsto B({\bf u})$ is a linear map and all $\cb_{B_i}$ are nonnegative then $x^{B({\bf u_i}^+)}|x^{B({\bf u}^+)}$. Since divisibility is compatible to any monomial order, then via Remark~\ref{B(u)}, we get that $x^{B({\bf u})^+}\in (x^{B({\bf u_1})^+}, \ldots,x^{B({\bf u_t})^+})$ and thus $\mathcal G'$ is a Gr\"obner basis of $I_A$ with respect to $<_1$. Finally, to prove that $\mathcal G'$ is reduced we argue by contradiction. This implies that there exists an integer $i$ such that $x^{B({\bf u_i}^-)}\in\ini_{<_1}(I_A)$, so $x^{B({\bf u_i}^-)}$ is divisible by some $x^{B({\bf u_j}^+)}$. This in turn implies that $y^{{\bf u_j}^+}|y^{{\bf u_i}^-}$, a contradiction since $\mathcal G$ is reduced. Therefore we obtain $\mathcal G'$ is reduced Gr\"obner basis with respect to $<_1$, as desired.      

Conversely, let  $\mathcal{G}'=\{x^{B({\bf u_1})^+}-x^{B({\bf u_1})^-},\ldots,x^{B({\bf u_t})^+}-x^{B({\bf u_t})^-}\}$ be a reduced Gr\"obner basis of $I_A$ with respect to a monomial order $<$ on  $K[x_1,\ldots,x_n]$. Using similar arguments as above one can prove that the set $\mathcal G=\{y^{{\bf u_1}^+}-y^{{\bf u_1}^-},\ldots,y^{{\bf u_t}^+}-y^{{\bf u_t}^-}\}$ is a reduced Gr\"obner basis of $A_B$ with respect to the monomial order $<'$ on $K[y_1,\ldots,y_s]$ defined as follows
\[
y^{\bf u}<' y^{\bf v} \ \text{ if and only if } \ x^{B({\bf u})}< x^{B({\bf v})}.
\]  
That $<'$ is a monomial order follows easily once we note that $<'$ is well defined, since all $\cb_{B_i}$'s are nonnegative. Therefore the proof is complete. \qed

\end{proof}

In summary, Theorem~\ref{gen_all_is_well} can be combined with Propositions~\ref{stable_markov}, \ref{stable_indispensable}, \ref{stable_ugb}  to provide, in particular, justification for the terminology `stable' toric ideals:

\begin{Theorem}\label{stable_toric}
Let $I_A$ be a stable toric ideal. Then the bijective correspondence between the elements of $\Ker_{\ZZ}(A)$ and $\Ker_{\ZZ}(A_B)$ given by ${\bf u}\mapsto B({\bf u})$, is preserved when we restrict to any of the following sets: Graver basis, circuits, indispensable binomials, minimal Markov bases, reduced Gr\"obner bases (universal Gr\"obner basis).
\end{Theorem}

\begin{Example}\label{infinite_not_mixed}
{\em Based on Theorem~\ref{inverse_construction}, there are infinitely many stable toric ideals. In fact, to construct them  is enough to consider matrices $A$ obtained as in Theorem~\ref{inverse_construction}, starting from arbitrary vectors $\ab_i$'s, but considering only vectors $\cb_i$'s with positive coordinates. Then via Remark~\ref{B(u)}, the corresponding subbouquets of $A$ are either free or non-mixed, and thus $I_A$ is a stable toric ideal.}
\end{Example}


We conclude this section with an application of the stable toric ideals to the construction of generic toric ideals, an open problem posed by Miller and Sturmfels \cite[Section 9.4]{MS}. Recall from \cite{PS} that a toric ideal $I_A$ is called {\em generic} if it is minimally generated by binomials with full support, that is $I_A=(x^{\bf u_1}-x^{\bf v_1},\ldots,x^{\bf u_s}-x^{\bf v_s})$, and none of the vectors ${\bf u_i}-{\bf v_i}$ has a zero coordinate. The following result states that in the case of stable toric ideals genericity is preserved when passing from $A$ to $A_B$ and conversely.

\begin{Theorem}\label{generic_preserved}
Let $I_A$ be a stable toric ideal. Then $I_A$ is a generic toric ideal if and only if $I_{A_B}$ is a generic toric ideal.   
\end{Theorem}
\begin{proof}
Applying Theorem~\ref{stable_toric} we know that the map ${\bf u}\mapsto B({\bf u})$ induces a bijective correspondence between the minimal Markov bases. Thus, since $B((u_1,\ldots,u_s))=\sum_{i=1}^s \cb_{B_i}u_i$ and $\cup_i \Supp(\cb_{B_i})=[n]$, then $\MM$ is a minimal Markov basis of $A_B$, with each vector having full support if and only if $\{B({\bf u}): \ {\bf u}\in\MM\}$ is a minimal Markov basis of $A$ with each vector having full support. \qed  

\end{proof}

\begin{Remark}\label{generic_construct}
{\em In particular, Section 2 provides a way to construct an infinite class of generic toric ideals starting from an arbitrary example of a generic toric ideal. More precisely, one can use the examples of generic toric ideals, see \cite[Example 2.3, Theorem 2.4]{Oj}, and for each such example choose arbitrary $\cb_i$'s with all coordinates positive to construct matrices $A$ following the procedure given in Theorem~\ref{inverse_construction}.  Theorem~\ref{generic_preserved} guarantees that $I_A$ will be a generic toric ideal. }
\end{Remark}

\begin{Theorem}  \label{all_homological_stable}
Let $I_A\subset S=K[x_1,\ldots,x_n]$ be a stable positively graded toric ideal and $A_B=[\ab_{B_1},\ldots,\ab_{B_s}]$ with $I_{A_B}\subset R=K[y_1,\ldots,y_s]$. If we denote by $\FF_{\bullet}$ the minimal $\NN A_B$-graded free resolution of $R/I_{A_B}$, then $\FF_{\bullet}\otimes_{R}S$ is a minimal $\NN A$-graded free resolution of $I_{A}$.     
\end{Theorem}
\begin{proof}
Since $I_A$ is a stable toric ideal then all the non-free bouquets are non-mixed. Without loss of generality, we may assume, as in the proof of Theorem~\ref{gen_iso_kernels} that $A$ has no free bouquet and there exist integers $1\leq i_1<i_2<\cdots<i_{s-1}\leq n$ such that $\ab_{1},\ldots,\ab_{i_1}$ belong to the bouquet $B_1$, and so on, until $\ab_{i_{s-1}+1},\ldots,\ab_n$ belong to the bouquet $B_s$. In particular, by Remark~\ref{B(u)}, all nonzero coordinates of $\cb_{B_1},\ldots,\cb_{B_s}$ are positive. We define the following $K$-algebra homomorphism
\begin{eqnarray*}
\phi: R&\to& S \\
y_1&\mapsto& x_1^{(c_{B_1})_1}\cdots x_{i_1}^{(c_{B_1})_{i_1}}\\
 &\vdots& \\
y_s&\mapsto& x_{i_{s-1}+1}^{(c_{B_s})_{i_{s-1}+1}}\cdots x_{n}^{(c_{B_s})_{n}}.
\end{eqnarray*}
The homomorphism $\phi$ is well-defined since all nonzero coordinates of $\cb_{B_i}$ are positive, and is a graded homomorphism with respect to the gradings induced by 1) $\NN A_B$ on $R$, that is $\deg(y_i)=\ab_{B_i}$  for all $i=1,\ldots,s$, and 2) $\NN A$ on $S$, that is $\deg(x_j)=\ab_j$ for all $j=1,\ldots,n$. Indeed,  
\[
\deg(y_j)=\ab_{B_j}=\sum_{(c_{B_j})_k\neq 0} (c_{B_j})_k \ab_k = \sum_{k=i_{j-1}+1}^{i_j} (c_{B_j})_k\ab_k = \deg (x_{i_{j-1}+1}^{(c_{B_j})_{i_{j-1}+1}}\cdots x_{n}^{(c_{B_j})_{i_{j}}}), 
\]
the last one being the degree $\deg(\phi(y_j))$ for all $j$. In addition,   
\[
I_A=(x^{B({\bf u})^+}-x^{B({\bf u})^-}: {\bf u}\in\Ker_{\ZZ}(A_B))=(x^{B({\bf u}^+)}-x^{B({\bf u}^-)}: {\bf u}\in\Ker_{\ZZ}(A_B)),
\]   
where the first equality follows from Theorem~\ref{gen_iso_kernels} and the second equality from Remark~\ref{B(u)}, which implies $\phi(I_{A_B})=I_A$, since by definition $\phi(y^{{\bf u}^+}-y^{{\bf u}^-})=x^{B({\bf u}^+)}-x^{B({\bf u}^-)}$ for all ${\bf u}\in\Ker_{\ZZ}(A_B)$.   

Applying now \cite[Theorem 18.16]{Ei} we obtain that $\phi$ is flat. This implies that the natural map $I_{A_B}\otimes_{R} S\rightarrow I_A$ is an isomorphism of graded $S$-modules; see \cite[Proposition 6.1]{Ei}. Finally, flatness of $\phi$ ensures that tensoring the minimal graded free resolution of $I_{A_B}$ as $R$-module with $S$ we obtain the minimal graded free resolution of $I_A$ as $S$-module, as desired.  \qed 

\end{proof}

\begin{Remark}\label{stable_robust}
{\em Let $I_A$ be a stable toric ideal. It follows from Theorem~\ref{stable_toric} that $I_A$ is a robust toric ideal if and only if $I_{A_B}$ is a robust toric ideal. In particular, using the same strategy described above for generic lattice ideals, we can construct robust toric ideals that are different from the ones considered in \cite{BR,BBDLMNS} which, in fact, correspond to toric ideals of graphs and toric ideals generated by degree two binomials. Using again Theorem~\ref{stable_toric} we also have that $I_A$ is generalized robust toric ideal if and only if $I_{A_B}$ is a generalized robust toric ideal (see \cite{T} for the definition of generalized robust toric ideal).}  

\end{Remark}

\begin{Remark}\label{first_stable}
\rm As a concluding remark of this section, given that stability preserves a lot of information, we discuss the case when a given ideal is not stable but we still wish to preserve all the combinatorial and algebraic information. To that end, suppose that $I_A$ is stable with respect to a certain subbouquet decomposition,  say $B_1,\ldots,B_t$. Then, the subbouquet ideal $I_{A'_B}$ associated to this decomposition preserves all ``combinatorial data" of $I_A$, and its minimal graded free resolution can be read off from the one of $I_A$. Indeed, the reader will have noted that, throughout this section, the only part of the stability hypothesis on $I_A$ that is relevant for the proofs is the fact that the bouquets are non-mixed. Maximality of the bouquets was not assumed. Hence, the results hold true for any subbouquet decomposition. Therefore, in case of a non-stable toric ideal, we are motivated to look for a ``maximal" subbouquet decomposition such that toric ideal is stable with respect to it. There is indeed such a subbouquet decomposition and a canonical way to obtain it. Since $I_A$ is not stable, there exists at least one mixed bouquet. Each mixed bouquet has a subbouquet decomposition into two maximal non-mixed subbouquets. To see this let us consider a mixed bouquet $B$, which by definition contains an edge $\{\ab_i,\ab_j\}\in E_A^-$. By Lemma~\ref{bouq_check} we have $G(\ab_i)=\lambda G(\ab_j)$ for some $\lambda<0$. We define $B_1$ to be the clique induced on the subset of vertices of $B$ for which the Gale transform is a positive multiple of $G(\ab_i)$, while $B_2$ is the clique induced on the subset of vertices of $B$ for which the Gale transform is a positive multiple of $G(\ab_j)$. It is straightforward to see that $B_1$ and $B_2$ are the desired maximal non-mixed subbouquets. Thus considering the non-mixed bouquets of $I_A$ and taking all maximal non-mixed subbouquets of all mixed bouquets we obtain the desired canonical subbouquet decomposition. Note also that this subbouquet decomposition is non-trivial only if at least one of the non-mixed subbouquets is not an isolated vertex. This canonical subbouquet decomposition provides  the subbouquet ideal $I_{A'_B}$. Finally, we can pass from $I_{A'_B}$ to $I_{A_B}$ through a bouquet graph whose non-mixed bouquets are isolated vertices, and mixed bouquets (if any) are singleton edges.
\end{Remark}     

The following example explains the general strategy of passing from a toric ideal to its bouquet ideal through a subbouquet ideal, where the subbouquet is chosen  such that the toric ideal is stable with respect to the underlying subbouquet decomposition.
               
\begin{Example}\label{general_decomposition} 
{\em  Let $A=[\ab_1,\ldots,\ab_9]$ be the integer matrix
\[
\left( \begin{array}{ccccccccc}
3 & 0 & 0 & 0 & 4 & 5 & 0 & 0 & 0\\
1 & 0 & 1 & 0 & 0 & 0 & 0 & 0 & 0\\
0 & 0 & 0 & 0 & 0 & 1 & 0 & 1 & 0\\
-1 & 1 & 0 & 0 & 0 & 0 & 0 & 0 & 0\\
0 & 0 & -1 & 1 & 0 & 0 & 0 & 0 & 0\\
0 & 0 & 0 & 0 & 0 & -1 & 1 & 0 & 0\\
0 & 0 & 0 & 0 & 0 & 0 & 0 & -1 & 1
\end{array} \right).
\] 
It has three non-free bouquets, $B_1,B_2,B_3$, with $B_1,B_3$ mixed and $B_2$ non-mixed, hence $I_A$ is not stable. The bouquet $B_1$ is on the set of vertices $\{\ab_1,\ab_2,\ab_3,\ab_4\}$  with $G(\ab_1)=G(\ab_2)=-G(\ab_3)=-G(\ab_4)$, $B_2$ is the isolated vertex $\ab_5$, and $B_3$ has the set of vertices $\{\ab_6,\ab_7,\ab_8,\ab_9\}$ with $G(\ab_6)=G(\ab_7)=-G(\ab_8)=-G(\ab_9)$. Each of the mixed bouquets $B_1,B_3$ has two maximal non-mixed subbouquets $B'_1,B''_1$ and $B'_3,B''_3$, respectively.   Moreover, the encoding vectors are $\cb_{B'_1}=(1,1,0,0,0,0,0,0,0)$, $\cb_{B''_1}=(0,0,1,1,0,0,0,0,0)$, $\cb_{B_2}=(0,0,0,0,1,0,0,0,0)$, $\cb_{B'_3}=(0,0,0,0,0,1,1,0,0)$ and $\cb_{B''_3}=(0,0,0,0,0,0,0,1,1)$. Therefore, the associated subbouquet matrix $A'_B=[\ab_{B'_1},\ab_{B''_1},\ab_{B_2},\ab_{B'_3},\ab_{B''_3}]\in\ZZ^{7\times 5}$ is 
\[
\left( \begin{array}{ccccc}
3 & 0 & 4 & 5 & 0\\
1 & 1 & 0 & 0 & 0\\
0 & 0 & 0 & 1 & 1\\
{\bf 0} & {\bf 0} & {\bf 0} & {\bf 0} & {\bf 0}
\end{array} \right),
\]     
where the last row of bold zeros stands for the zero block matrix from $\ZZ^{4\times 5}$. Since $I_A$ is stable with respect to the subbouquet decomposition given by the subbouquets $B'_1,B''_1,B_2,B'_3,B''_3$, this implies that the conclusions of Theorem~\ref{stable_toric} and Theorem~\ref{all_homological_stable} do apply, when restricted to $I_A$ and $I_{A'_B}$. In other words, when passing from $I_A$ to $I_{A'_B}$  we preserve all the combinatorial and algebraic data. Finally, the matrix $A'_B$ has three non-free bouquets, two mixed and one non-mixed, with the two mixed consisting of a single edge $\{\ab_{B'_1},\ab_{B''_1}\}$ and $\{\ab_{B'_3},\ab_{B''_3}\}$, respectively, while the non-mixed one is the isolated vertex $\ab_{B_2}$. Computing the bouquet matrix of $A'_B$ we obtain  
\[
A_B= \left( \begin{array}{ccccc}
3 & 4 & 5\\
{\bf 0} & {\bf 0} & {\bf 0}
\end{array} \right)\in\ZZ^{7\times 3},
\]   
and the toric ideal of $A_B$ equals the toric ideal of the monomial curve $(3 \ 4 \ 5)$. Comparing now $I_A$ (and thus $I_{A'_B}$) with $I_{A_B}$ we note, for example, that a minimal Markov basis of $I_A$ has six elements, while a minimal Markov basis of $I_{A_B}$ has only three. 
} 
\end{Example}

\section{ A combinatorial characterization of strongly robust toric ideals } 
 \label{sec:EqualBases}

It is well known from \cite[Proposition 7.1]{St} that the Graver basis, the universal Gr{\"o}bner basis, any reduced Gr{\"o}bner basis and any minimal Markov basis are equal for the toric ideal of the second Lawrence lifting of an arbitrary integer matrix. It is also well known that Lawrence liftings are not the only matrices with this property. For example, it was  shown to hold for $2$-regular uniform hypergraphs by Gross and Petrovi\'c in \cite{GP}, and for robust toric ideals of graphs by Boocher et al.\ in \cite{BBDLMNS}. Furthermore, such examples can have both mixed and non-mixed bouquets, as in Example~\ref{four_sets}(b). 
Therefore, it is clear that  bouquets alone do not capture equality of  bases. 

The main result of this section, Theorem~\ref{graver=indisp}, is a characterization of  toric ideals  whose bases are equal. It relies on two additional ingredients. The first one is the familiar notion of a semiconformal decomposition (cf. Section~\ref{sec:OnStableToricIdeals}), which provides an algebraic characterization of equality of bases. The second one is the new concept of \emph{$S$-Lawrence} ideals that generalizes the classical second Lawrence lifting.

	Before we proceed,  recall that  if	$\Ker_{\ZZ}(A)\cap\NN^n\neq\{\bf 0\}$, the four sets of bases can never be simultaneously equal	by \cite[Theorem 4.18]{CTV1}.
 Thus we may assume for the rest of this section that $\Ker_{\ZZ}(A)\cap\NN^n=\{\bf 0\}$, which is equivalent to saying that $I_A$ is positively graded. 
 Recall also that the  {\em fiber} $\mathcal{F}_{{\bf u}}$ of a monomial $x^{\bf u}$ is the set $\{{\bf t}\in\NN^n : {{\bf u}}-{\bf t}\in \Ker_{\ZZ}(A)\}$. When $I_A$ is positively graded, $\mathcal{F}_{{\bf u}}$ is a finite set.

 For $S\subset [n]$ and a vector ${\bf u}=(u_1,\ldots,u_n)\in \ZZ^n$,  we define the \emph{$S$-part of ${\bf u}$} to be the vector $(u_i|i\in S)$ with $|S|$ coordinates; if $S=\emptyset$, then the $S$-part of the vector ${\bf u}$ is $0$. Define the \emph{$S$-parallelepiped of ${\bf u}$}, $P_S({\bf u})$, to be the cartesian product $\prod_{i\in S}[0,\max\{u_i^+, u_i^-\}]$ if $S\neq\emptyset$, and set $P_{\emptyset}({\bf u})=\{0\}$.
 
\begin{Definition}\label{S-Lawrence}
{\em Fix a subset $S$ of $[n]$. The toric ideal $I_A$ is said to be {\em $S$-Lawrence } if and only if for every element ${\bf u}\in\Gr(A)$, there exists no element ${\bf w}$ in the fiber of ${\bf u}^+$, different from ${\bf u}^+$ and ${\bf u}^-$, such that the $S$-part of ${\bf w}$ belongs to $P_S({\bf u})$.}
\end{Definition} 

 The $S$-Lawrence property is a natural one. Namely, by (the conformal) definition  of the Graver basis, every toric ideal is $S$-Lawrence in the case when $S=[n]$. Another straightforward property is that if $I_A$ is $S$-Lawrence and $S\subset S'$ then $I_A$ is also $S'$-Lawrence. On the other hand, if $S=\emptyset$ and $I_A$ is $\emptyset$-Lawrence, then for every ${\bf u}\in\Gr(A)$ the fiber of ${\bf u}^+$ consists of just two elements ${\bf u}^+,{\bf u}^-$. In particular, the fiber of ${\bf u}^+$ being finite implies that all fibers of $I_A$ are finite (see for example \cite[Proposition 2.3]{CTV1}), and thus $I_A$ is positively graded. Since the fiber of each Graver basis element ${\bf u}$ consists of the two elements ${\bf u}^+,{\bf u}^-$, whose supports are disjoint,  every Graver basis element is in a minimal Markov basis, and thus indispensable via \cite[Corollary 2.10]{CKT}. Hence  the two bases are equal:  $\Gr(A)=\MS(A)$. In summary, $I_A$ is $\emptyset$-Lawrence if and only if $I_A$ is positively graded and  $I_A$ is strongly robust. By \cite[Theorem 4.18]{CTV1} if $I_A$ is not positively graded then $\MS(A)=\emptyset$ and thus $\MS(A)\neq \Gr(A)$, implying that $I_A$ can not be strongly robust. Therefore, we have $I_A$ is $\emptyset$-Lawrence if and only if $I_A$ is strongly robust. 

We are now ready to state the desired  combinatorial characterization. 

\begin{Theorem}\label{graver=indisp}
Let $B_1, \dots, B_s$ be the bouquets of $A=[{\bf a}_1,\dots,{\bf a}_n]$, and define $A_B=[\ab_{B_1},\dots,\ab_{B_s}]$. Let $S\subset [s]$ be the subset of coordinates corresponding to the mixed bouquets. Then the following are equivalent:
\begin{enumerate}
\item [(a)] There exists no element in the Graver basis of $A_B$ which has a proper semiconformal decomposition that is conformal on the coordinates corresponding to $S$.  
\item [(b)]  $I_A$ is strongly robust,  i.e. the following sets coincide:
\begin{itemize}
\item the Graver basis of $A$,
\item the universal Gr{\"o}bner basis of $A$,
\item any reduced Gr{\"o}bner basis of $A$,
\item any minimal Markov basis of $A$.
\end{itemize}    
\item [(c)] The toric ideal of $A_B$ is $S$-Lawrence.  
\end{enumerate}
\end{Theorem}

\begin{proof}
We may assume that the $s$ bouquets of $A$ are not free and partition the set $\{\ab_1,\ldots,\ab_n\}$ into $s$ subsets such that the first $i_1$ vectors belong to the bouquet $B_1$, the next $i_2$ vectors belong to the bouquet $B_2$, and so on, the last $i_s$ vectors belong to the bouquet $B_s$. 

In order to prove the theorem we analyze two cases: 1) $S=\emptyset$ and 2) $S\neq\emptyset$. If  $S=\emptyset$, then the hypothesis of Theorem~\ref{graver=indisp} translates to $I_A$ being a stable, positively graded toric ideal. It follows from the remarks after Definition~\ref{S-Lawrence} that condition (c) is equivalent to $I_{A_B}$ positively graded and $\Gr(A_B)=\MS(A_B)$. On the other hand, since $I_A$ is a stable, positively graded toric ideal,  Lemma~\ref{mixed_same_positive} implies that $I_{A_B}$ is also positively graded. Thus (a) is equivalent, via \cite[Proposition 1.1]{CTV}, to $I_{A_B}$ being  positively graded and $\Gr(A_B)=\MS(A_B)$. Therefore, (a) and (c) are equivalent. Note that equivalence of (a) with (b) follows from Theorem~\ref{stable_toric}.     

If $S\neq\emptyset$, we assume that $B_1,\ldots,B_t$ are the mixed bouquets of $A$ for some $t$ with $t\leq s$. We will show the equivalence of (a) and (b), and then of (a) and (c). 

$(b)\Rightarrow (a):$ Since the four sets coincide and $I_A$ is positively graded, $\Gr(A)=\MS(A)$. Assume by contradiction that there exists a Graver basis element ${\bf u}$ of $A_B$ which has a proper semiconformal decomposition that is conformal on the components corresponding to $S$. This implies that there exist nonzero vectors ${\bf v}, {\bf w}\in\Ker_{\ZZ}(A_B)$ such that ${\bf u}={\bf v}+_{sc}{\bf w}$ and $(u_1,\ldots,u_t)=(v_1,\ldots,v_t)+_{c}(w_1,\ldots,w_t)$. Since ${\bf u}\in\Gr(A_B)$ it follows from Theorem~\ref{gen_all_is_well} that $B({\bf u})\in\Gr(A)$. We will prove that $B({\bf u})=B({\bf v})+_{sc}B({\bf w})$, a contradiction to our assumption that $\Gr(A)=\MS(A)$. Note that conformality implies that $v_i$ and $w_i$ have the same sign for all $i=1,\ldots,t$ thus 
\begin{eqnarray}\label{conf}
({\bf c}_{B_1}u_1,\ldots,{\bf c}_{B_t}u_t)=({\bf c}_{B_1}v_1,\ldots,{\bf c}_{B_t}v_t)+_{c}({\bf c}_{B_1}w_1,\ldots,{\bf c}_{B_t}w_t).
\end{eqnarray}
Note that $t<s$, for otherwise we obtain $B({\bf u})=B({\bf v})+_{c}B({\bf w})$ and thus $B({\bf u})\notin\Gr(A)$, a contradiction. Since the bouquets $B_{t+1},\ldots,B_s$ are not mixed, it follows that the vectors ${\bf c}_{B_{t+1}},\ldots,{\bf c}_{B_s}$ have all coordinates positive and thus 
\begin{eqnarray}\label{semiconf}
({\bf c}_{B_{t+1}}u_{t+1},\ldots,{\bf c}_{B_s}u_s)=({\bf c}_{B_{t+1}}v_{t+1},\ldots,{\bf c}_{B_s}v_s)+_{sc}({\bf c}_{B_{t+1}}w_{t+1},\ldots,{\bf c}_{B_s}w_s).
\end{eqnarray} 
Combining (\ref{conf}) and (\ref{semiconf}) we obtain $B({\bf u})=B({\bf v})+_{sc}B({\bf w})$, and the claim follows. 

$(a)\Rightarrow (b)$: Assume that there exists no element in the Graver basis of $A_B$ which has a proper semiconformal decomposition that is conformal on the components corresponding to $S$. We argue by contradiction and suppose that $\MS(A)$ is properly contained in $\Gr(A)$. Then, by Theorem~\ref{gen_all_is_well}, there exists an element $B({\bf u})\in\Gr(A)$ for some ${\bf u}\in\Gr(A_B)$ such that $B({\bf u})\notin\MS(A)$. Applying again \cite[Proposition 1.1]{CTV} we obtain that there exist nonzero vectors ${\bf v},{\bf w}\in\Ker_{\ZZ}(A_B)$ such that $B({\bf u})=B({\bf v})+_{sc}B({\bf w})$. We will prove that ${\bf u}={\bf v}+_{sc}{\bf w}$ and $(u_1,\ldots,u_t)=(v_1,\ldots,v_t)+_{c}(w_1,\ldots,w_t)$, a contradiction to our hypothesis. Since $B_1,\ldots,B_t$ are mixed each one of the vectors ${\bf c}_{B_1},\ldots,{\bf c}_{B_t}$ has one negative and one positive coordinate by Lemma~\ref{mixed_bouquet}. Since $B({\bf u})=B({\bf v})+_{sc}B({\bf w})$, it follows that 
\[
({\bf c}_{B_1}u_1,\ldots,{\bf c}_{B_t}u_t)=({\bf c}_{B_1}v_1,\ldots,{\bf c}_{B_t}v_t)+_{sc}({\bf c}_{B_1}w_1,\ldots,{\bf c}_{B_t}w_t),
\]   
and thus $v_i$ and $w_i$ have the same sign for all $1\leq i\leq t$. Therefore we obtain that $(u_1,\ldots,u_t)=(v_1,\ldots,v_t)+_{c}(w_1,\ldots,w_t)$. If $t=s$ then the proof of this implication is complete.
	If $t<s$, then bouquets $B_{t+1},\ldots,B_s$ are not mixed, and consequently all coordinates of the vectors  ${\bf c}_{B_{t+1}},\ldots,{\bf c}_{B_s}$ are positive. Since 
\[
({\bf c}_{B_{t+1}}u_{t+1},\ldots,{\bf c}_{B_s}u_s)=({\bf c}_{B_{t+1}}v_{t+1},\ldots,{\bf c}_{B_s}v_s)+_{sc}({\bf c}_{B_{t+1}}w_{t+1},\ldots,{\bf c}_{B_s}w_s),
\]
it follows  that $(u_{t+1},\ldots,u_s)=(v_{t+1},\ldots,v_s)+_{sc}(w_{t+1},\ldots,w_s)$. This implies that ${\bf u}={\bf v}+_{sc}{\bf w}$ and the proof is complete. 

$(a)\Rightarrow (c)$: Assume by contradiction that $I_{A_B}$ is not $S$-Lawrence. This implies that there exists a vector ${\bf u}\in\Gr(A_B)$ and a positive vector ${\bf w}\in\MF_{{\bf u}^+}\setminus\{{\bf u}^+, {\bf u}^-\}$ such that the $S$-part of ${\bf w}$ belongs to $P_S({\bf u})$. We claim that (*): ${\bf u}=({\bf u}^+-{\bf w})+_{sc}({\bf w}-{\bf u}^-)$ and the sum is conformal on the coordinates corresponding to $S$, which leads us to a contradiction and the claim follows. We will prove first that the sum is conformal on the coordinates of $S$. Let $i$ be an integer of $S$. If $u_i>0$ then  $u_i=u^+_i$, $u^-_i=0$, and by the hypothesis $w_i\leq u^+_i$. Thus  $u^+_i-w_i\geq 0$, $w_i-u^-_i=w_i\geq 0$. If $u_i=0$ then $u^+_i=u^-_i=0$ and by hypothesis also $w_i=0$, so $ u^+_i-w_i=w_i-u^-_i=0$. Finally, if $u_i<0$ then  $u_i=u^-_i$, $u^+_i=0$, $w_i\leq  u^-_i$, $u^+_i-w_i=-w_i\leq 0$, $w_i- u^-_i\leq 0$. Hence it follows that the sum (*) is conformal on the coordinates corresponding to $S$, and implicitly the sum (*) is also semiconformal on the coordinates corresponding to $S$. It remains to check  semiconformality only on the coordinates not in $S$. There are such coordinates, since otherwise we obtain that the sum is conformal, a contradiction to ${\bf u}\in\Gr(A_B)$. Let $i\not\in S$ be an integer. If $u^+_i-w_i>0$ then $u^+_i>0$, $u^-_i=0$, and thus  $w_i- u^-_i= w_i\geq 0$. On the other hand if $ w_i- u^-_i<0$ then $ u^-_i>0$, $ u^+_i=0$ and we obtain $u^+_i- w_i=- w_i\leq 0$. Therefore we have proved that the sum (*) is semiconformal, and the claim follows. 

$(c)\Rightarrow (a)$:  Assume by contradiction that there exists an element ${\bf u}\in\Gr(A_B)$ which has a proper semiconformal decomposition, that is conformal on the coordinates corresponding to $S$. This implies that there exist nonzero vectors ${\bf v},{\bf w}\in\Ker_{\ZZ}(A_B)$ such that ${\bf u}={\bf v}+_{sc}{\bf w}$ and $(u_1,\ldots,u_t)=(v_1,\ldots,v_t)+_{c}(w_1,\ldots,w_t)$. By definition, semiconformality implies that ${\bf u}^+\geq {\bf v}^+$ and ${\bf u}^-\geq {\bf w}^-$. We claim that the vector ${\bf z}={\bf u}^+-{\bf v}={\bf u}^-+{\bf w}$ is positive, and thus ${\bf z}\in\MF_{{\bf u}^+}$, ${\bf z}$ is different from ${\bf u}^+$ and ${\bf u}^-$ and the $S$-part of ${\bf z}$ belongs to $P_S({\bf u})$. This yields a contradiction to the hypothesis that $I_{A_B}$ is $S$-Lawrence, and the proof is complete. Since ${\bf z}={\bf u}^+-{\bf v}={\bf u}^+-{\bf v}^++{\bf v}^-$, then ${\bf z}\geq {\bf 0}$ follows from ${\bf u}^+\geq {\bf v}^+$. Moreover, ${\bf v},{\bf w}\neq{\bf 0}$, provides that ${\bf z}\neq {\bf u}^+,{\bf u}^-$. Finally, that the $S$-part of ${\bf z}$ belongs to $P_S({\bf u})$  holds immediately after analyzing the cases 1) $i\in S$ with $u_i>0$, 2) $i\in S$ with $u_i=0$, and 3) $i\in S$ with $u_i<0$.\qed

\end{proof}

Theorem~\ref{graver=indisp}  expresses the equality of the four sets of bases in terms of the bouquet structure and the $S$-Lawrence property. It mainly provides a way of producing examples of  strongly robust toric ideals: choose any toric ideal $I_B\subset K[x_1,\dots ,x_s]$ and a subset $S$ of $[s]$ such that $I_B$ is $S$-Lawrence, then use Section~\ref{sec:GeneralizedLawrenceMatrices} to assign $s$ bouquets to $I_B$ in such a way that the ones assigned to the elements in $S$ are mixed, and obtain finally a toric ideal $I_A$ for which the four sets of bases are equal. However, Theorem~\ref{graver=indisp} says that this procedure provides all strongly robust toric ideals, and a better understanding of the $S$-Lawrence property in general would provide a complete characterization of these toric ideals. In the following example we show how to use the equivalence of (a) and (c) from Theorem~\ref{graver=indisp}.

\begin{Example}\label{four_sets}
{\em a)  Consider the hypergraph $\MH=(V,\ME)$ with the set of vertices $V=\{x,v_1,\ldots,v_{14}\}$ and whose set of edges $\ME$ consists of the following $12$ edges: $E_1=\{x, v_1, v_2\}$, $E_2=\{x, v_3, v_4\}$, $E_3=\{x, v_5, v_6\}$, $E_4=\{ v_1, v_3, v_5\}$, $E_5=\{v_2, v_4, v_6\}$, $E_6=\{x, v_7, v_8\}$, $E_7=\{x, v_9, v_{10}\}$, $E_8=\{x, v_{11}, v_{12}\}$, $E_9=\{ x, v_{13}, v_{14}\}$, $E_{10}=\{v_7, v_8, v_9\}$,  $E_{11}=\{v_{10}, v_{11}, v_{13}\}$, $E_{12}=\{v_{12}, v_{14}\}$. Let $A=[{\bf a}_1,\ldots,{\bf a}_{13}]\in\ZZ^{15\times 13}$, where ${\bf a}_1,\ldots,{\bf a}_{12}$ are the support vectors of the edges $E_1,\ldots,E_{12}$ and ${\bf a}_{13}$ is the vector $(5,0,\ldots,0)$. With respect to the basis $G_1,G_2$ of $\Ker_{\ZZ}(A)$, where
\[
G_1=(0,0,0,0,0,5,5,5,5,-5,-5,-5,-4),
\]
and 
\[
G_2=(1,1,1,-1,-1,-2,-2,-2,-2,2,2,2,1),
\]    
we see that the Gale transforms are: $G({\bf a}_1)=G({\bf a}_2)=G({\bf a}_3)=-G({\bf a}_4)=-G({\bf a}_5)=(0,1)$, $G({\bf a}_6)=\cdots=G({\bf a}_9)=-G({\bf a}_{10})=\cdots=-G({\bf a}_{12})=(5,-2)$, and $G({\bf a}_{13})=(-4,1)$. Thus the bouquet graph $G_A$ has three non-free bouquets $B_1,B_2,B_3$: The first two, $B_1$ and $B_2$, are mixed bouquets corresponding to the sets of vertices  $\{{\bf a}_1,\ldots,{\bf a}_{5}\}$ and $\{{\bf a}_6,\ldots,{\bf a}_{12}\}$, respectively. The third, $B_3$, is an isolated vertex ${\bf a}_{13}$. Hence $A_B=[{\bf a}_{B_1},{\bf a}_{B_2},{\bf a}_{B_3}]$, where ${\bf a}_{B_1}=(3,0,\ldots,0)$, ${\bf a}_{B_2}=(4,0,\ldots,0)$, ${\bf a}_{B_3}=(5,0,\ldots,0)$, $s=3$, $S=\{1,2\}$. The Graver basis elements of $A_B$ are $(4,-3,0)$, $(1,-2,1)$, $(3,-1,-1)$, $(2,1,-2)$, $(5,0,-3) $, $(1,3,-3)$, $(0,5,-4)$. Since $(4,-3,0)=(3,-1,-1)+_{sc}(1,-2,1)$, and the sum is conformal on the first two coordinates, i.e. on $S$, it follows that the condition (a) of Theorem~\ref{graver=indisp} is not satisfied and therefore the four sets are not simultaneously equal. Note that in this case the toric ideal of $A_B$ is positively graded.  

b) Consider the graph from Figure~\ref{Figa}, and denote by $A$ its incidence matrix whose column vectors ${\bf e}_1,\ldots,{\bf e}_{15}$ are given by the support vectors of the edges $e_1,\ldots,e_{15}$. 

\begin{figure}[hbt]
\begin{center}
\psset{unit=0.9cm}
\begin{pspicture}(0,0)(10,7)
\rput(1,2){$\bullet$}
\rput(3,1){$\bullet$}
\rput(3,3){$\bullet$}
\rput(5,2){$\bullet$}
\rput(7,1){$\bullet$}
\rput(9,2){$\bullet$}
\rput(7,3){$\bullet$}
\rput(4,4){$\bullet$}
\rput(5,6){$\bullet$}
\rput(6,4){$\bullet$}
\rput(5,1.65){$1$}
\rput(3,0.65){$2$}
\rput(0.65,2){$3$}
\rput(3,3.35){$4$}
\rput(3.65,4){$5$}
\rput(5,6.35){$6$}
\rput(6.35,4){$7$}
\rput(7,3.35){$8$}
\rput(9.35,2){$9$}
\rput(7,0.65){$10$}
\psline[linewidth=1.3pt, linecolor=black](1,2)(3,1)
\psline[linewidth=1.3pt, linecolor=black](1,2)(3,3)
\psline[linewidth=1.3pt, linecolor=black](5,2)(3,1)
\psline[linewidth=1.3pt, linecolor=black](5,2)(3,3)
\psline[linewidth=1.3pt, linecolor=black](5,2)(4,4)
\psline[linewidth=1.3pt, linecolor=black](5,2)(6,4)
\psline[linewidth=1.3pt, linecolor=black](4,4)(5,6)
\psline[linewidth=1.3pt, linecolor=black](6,4)(5,6)
\psline[linewidth=1.3pt, linecolor=black](5,2)(7,1)
\psline[linewidth=1.3pt, linecolor=black](5,2)(7,3)
\psline[linewidth=1.3pt, linecolor=black](9,2)(7,1)
\psline[linewidth=1.3pt, linecolor=black](9,2)(7,3)
\psline[linewidth=1.3pt, linecolor=black](5,2)(1,2)
\psline[linewidth=1.3pt, linecolor=black](5,2)(5,6)
\psline[linewidth=1.3pt, linecolor=black](5,2)(9,2)
\rput(4.1,1.25){${e_1}$}
\rput(2,1.25){${e_2}$}
\rput(1.9,2.75){${e_3}$}
\rput(4,2.75){${e_4}$}
\rput(3,2.25){${e_5}$}
\rput(4.15,3.1){${e_6}$}
\rput(4.15,5){${e_7}$}
\rput(5.85,5){${e_8}$}
\rput(5.85,3.1){${e_9}$}
\rput(4.65,4){${e_{10}}$}
\rput(5.9,2.75){${e_{11}}$}
\rput(8.1,2.75){${e_{12}}$}
\rput(8.2,1.25){${e_{13}}$}
\rput(6,1.25){${e_{14}}$}
\rput(7,2.25){${e_{15}}$}
\end{pspicture}
\end{center}
\caption{}\label{Figa}
\end{figure}
 
Then the bouquet graph associated to $A$ has nine non-free bouquets $B_1,\ldots,B_9$: the first six are mixed bouquets corresponding to the edges $\{{\bf e}_1,{\bf e}_2\}$, $\{{\bf e}_3,{\bf e}_4\}$, $\{{\bf e}_6,{\bf e}_7\}$, $\{{\bf e}_8,{\bf e}_9\}$, $\{{\bf e}_{11},{\bf e}_{12}\}$, $\{{\bf e}_{13},{\bf e}_{14}\}$, and the last three are the isolated vertices ${\bf e}_5,{\bf e}_{10},{\bf e}_{15}$. Therefore not all bouquets of $G_A$ are mixed, $A_B=\{{\bf a}_{B_1},\ldots,{\bf a}_{B_9}\}$, where ${\bf a}_{B_1}={\bf e}_1-{\bf e}_2$, ${\bf a}_{B_2}={\bf e}_3-{\bf e}_4$, ${\bf a}_{B_3}={\bf e}_6-{\bf e}_7$, ${\bf a}_{B_4}={\bf e}_8-{\bf e}_9$, ${\bf a}_{B_5}={\bf e}_{11}-{\bf e}_{12}$, ${\bf a}_{B_6}={\bf e}_{13}-{\bf e}_{14}$, ${\bf a}_{B_7}={\bf e}_5$, ${\bf a}_{B_8}={\bf e}_{10}$, ${\bf a}_{B_9}={\bf e}_{15}$ and $S=\{1,\ldots,6\}$. The toric ideal of $A_B$ is not positively graded, and thus all the fibers of $I_{A_B}$ are infinite. We will check that $I_{A_B}$ is $S$-Lawrence in order to use Theorem~\ref{graver=indisp} to conclude that the four bases are equal. The
 Graver basis of $A_B$ has fifteen elements, one of which is ${\bf u}=(0,1,1,0,0,0,-1,1,0)$. The degree of the fiber of ${\bf u}^+,{\bf u}^-$ is equal to ${\bf a}_{B_2}+{\bf a}_{B_3}+{\bf a}_{B_8}={\bf a}_{B_7}=(1,0,1,0,0,0,0,0,0,0)$. Since the fiber of ${\bf u}^+$ is found by computing all nonnegative solutions of the equation $\sum_{i=1}^9 n_i{\bf a}_{B_i}=(1,0,1,0,0,0,0,0,0,0)$, any element in the fiber is of one of the following three types of nonnegative vectors
\[
(\alpha,\alpha,\beta,\beta,\gamma,\gamma,1,0,0), \ (\alpha,\alpha+1,\beta,\beta-1,\gamma,\gamma,0,1,0), \ (\alpha,\alpha+1,\beta,\beta,\gamma,\gamma-1,0,0,1). 
\]    
Let ${\bf w}$ be in the fiber of ${\bf u}^+$ such that the $S$-part of ${\bf w}$ belongs to the $S$-parallelepiped of ${\bf u}$, that is $w_1=w_4=w_5=w_6=0$ and $w_2,w_3\leq 1$. It follows immediately that ${\bf w}$ is either of the first type with $\alpha=0,\beta=0,\gamma=0$, which implies ${\bf w}={\bf u}^-$, or is of the second type with $\alpha=0,\beta=1,\gamma=0$, and then ${\bf w}={\bf u}^+$. Hence there exists no vector ${\bf w}$ different than ${\bf u}^+$ and ${\bf u}^-$ which belongs to $P_S({\bf u})$. Similarly one can check for the rest of the elements of $\Gr(A_B)$ and conclude that $I_{A_B}$ is $S$-Lawrence. Therefore we obtain that the four sets are equal. The same conclusion could have been drawn for the toric ideal $I_A$ as a consequence of the graph-theoretical description of the Graver, universal Gr\"obner, and any minimal Markov basis of $A$ given in \cite{RTT,TT}.}
\end{Example}

As an immediate consequence of Theorem~\ref{graver=indisp} we obtain the following. 

\begin{Corollary}\label{gen_no_free}
Suppose that every non-free bouquet of $A$ is mixed. Then $I_A$ is strongly robust. 
\end{Corollary}
\begin{proof}
Note first that if every non-free bouquet is mixed then $\Ker_{\ZZ}(A)\cap\NN^n=\{\bf 0\}$. As before, we may assume that $A$ has no free bouquet and  since all  $s$ bouquets are mixed then  $S=\{1,\ldots,s\}$, and thus condition (a) of Theorem~\ref{graver=indisp} is trivially satisfied. An application of Theorem~\ref{graver=indisp} leads to the desired conclusion.  \qed

\end{proof}

Note that the condition that each non-free bouquet of $A$ is mixed is not a sufficient condition, as Example~\ref{four_sets}(b) shows. We close this section by showing how one recovers \cite[Theorem 7.1]{St} from Corollary~\ref{gen_no_free}.  

\begin{Remark}\label{second_lawrence}
{\em Let $D\in\ZZ^{m\times n}$ be an integer matrix and $\Lambda(D)$ its second Lawrence lifting. We denote  by $\alpha_1,\ldots,\alpha_{n}$ the columns of $D$. By Remark~\ref{why_gen_lawrence}, $\Lambda(D)$ is after a column permutation just the matrix $A$ constructed in Theorem~\ref{inverse_construction} for the vectors $\ab_i=\alpha_i$ and $\cb_i=(1,-1)$ for all $i=1,\ldots,n$. Applying Theorem~\ref{inverse_construction} we obtain that $A$ has $n$ subbouquets, which are either free or mixed by Lemma~\ref{mixed_bouquet}. Thus by Corollary~\ref{gen_no_free} we obtain the desired conclusion.} 
\end{Remark}


Interestingly, all previously known examples of equality $\Gr(A)=\MS(A)$, that is, of strongly robust toric ideals in the literature  not arising from graphs, as those from \cite{BBDLMNS}, and Example 4.3a), are  such that  matrices $A$ have only mixed bouquets. Our results provide many additional such easy-to-construct examples, e.g.\ Corollary~\ref{gen_no_free}, recovering, in particular, \cite[Proposition 7.1]{St} as a special case. On the other hand,  Example~\ref{four_sets} b) constructs a strongly robust toric ideal that has both mixed and non-mixed bouquets. However, we do not have an example of a strongly robust  toric ideal $I_A$, whose matrix $A$ has no mixed bouquet, raising  the natural question:
\begin{Question}\label{strongly_mixed}
 \emph{Is it true that if $I_A$ is strongly robust then $A$ has at least one mixed bouquet?} 
\end{Question}
For readers interested in this question, we gather a few remarks.  Note that if $A$ has $s$ bouquets then,  since  $I_{A_B}$ is always $[s]$-Lawrence the consequence of Theorem~\ref{graver=indisp} is that: in the case when $A$ has all of the bouquets mixed then the converse is also true, and thus $I_A$ is strongly robust. Recently, Sullivant showed in \cite{Su} that our question has a positive answer in the case of strongly robust toric ideals of codimension less than or equal to 2, based on the comprehensive knowledge of the codimension 2 lattice ideals from \cite{PS1}. The main ingredient in his proof was a complete characterization of the strongly robust codimension 2 toric ideals in terms of the Gale transform. In  general, we can only say that when $A$ has no mixed bouquets, Theorem~\ref{graver=indisp} implies that $I_A$ is strongly robust if and only if $I_{A_B}$ is strongly robust; also a consequence of Theorem~\ref{stable_toric}. 
Moreover, at the moment, we do not know whether $S$-Lawrence toric ideals have special properties when $\emptyset\subsetneq S\subsetneq [n]$.

\end{document}